\documentclass[reqno]{amsart}

\usepackage{appendix}
\usepackage[utf8]{inputenc}
\usepackage[T1]{fontenc}
\usepackage{lmodern}
\usepackage{color}
\usepackage{listings}
\usepackage{hyperref}
\usepackage{amsmath,bm}
\usepackage{amsthm,amssymb}
\usepackage{amsfonts,mathrsfs}
\usepackage{epstopdf}
\usepackage[table]{xcolor}
\usepackage{enumitem}
\usepackage[numbers,sort&compress]{natbib}
\usepackage{graphicx,pstricks,listings,subfigure,caption,tikz}
\usepackage{soul}
\usepackage{cleveref}
\usepackage{appendix}
\usepackage {setspace,caption}

\usepackage{geometry}
\allowdisplaybreaks[4]

\graphicspath{ {./my_report_images/} }

\sethlcolor{orange}

\numberwithin{equation}{section}
\newtheorem{corollary}{Corollary}[section]
\newtheorem{definition}{Definition}[section]
\newtheorem{lemma}{Lemma}[section]
\newtheorem{proposition}{Proposition}[section]
\newtheorem{remark}{Remark}[section]

\newtheorem{theorem}{Theorem}[section]

\begin{document}
\title[Well--Posedness of Periodic Strong Transonic Shocks]{On the Well--Posedness of Periodic Strong Transonic Shocks in Divergent Nozzles}

\author{Peng Qu}
\address[Peng Qu]{School of Mathematical Sciences, Fudan University, Shanghai, China/Shanghai Key Laboratory of Contemporary Applied Mathematics, Shanghai, China}
\email{pqu@fudan.edu.cn}

\author{Jiahui Wang}
\address[Jiahui Wang]{School of Mathematical Sciences, Fudan University, Shanghai, China}
\email{23110180037@m.fudan.edu.cn, corresponding author.}

\author{Huimin Yu}
\address[Huimin Yu]{Department of mathematics, Shandong Normal University, Jinan, China}
\email{hmyu@sdnu.edu.cn}

\author{Xiaomin Zhang}
\address[Xiaomin Zhang]{Department of mathematics, Shandong Normal University, Jinan, China}
\email{zxm15924687@163.com}

\date{}


\begin{abstract}
	This paper studies the $C^{1}$ existence and dynamical stability of temporal periodic solutions involving strong transonic shocks for the quasi--one--dimensional full compressible Euler equations in diverging nozzles. We reveal the physical structural dissipation mechanism of strong transonic shocks hidden in the Rankine--Hugoniot conditions. This allows us to estimate the resonance at the shock boundary caused by the left acoustic waves in the subsonic downstream region. With this dissipation, we complete the proof by developing a fraction--step linearized iterative method.
\end{abstract}
\maketitle

\section{Introduction}
\subsection{Setting of the problem}
We would like to study the transonic shock solutions for the following quasi--one--dimensional full compressible Euler equations:
\begin{equation}\label{a1}
	\left\{\begin{aligned}
		&\rho_{t}+(\rho u)_{x}=-\frac{a^{\prime}(x)}{a(x)}\rho u,\\
		&(\rho u)_{t}+(\rho u^{2}+p)_{x}=-\frac{a^{\prime}(x)}{a(x)}\rho u^{2},\\
		&(\rho E)_{t}+(\rho Eu+pu)_{x}=-\frac{a^{\prime}(x)}{a(x)}(\rho Eu+pu),
	\end{aligned}\right.
\end{equation}
where $\rho, u, p$ are the unknowns on $(t,x) \in \mathbb{R}_{+}\times[0,L]$, denoting the density, velocity and pressure of the gas respectively. The system \eqref{a1} describes flows of non--isentropic gas in a variable cross--section nozzle, where the function $a(x)$ stands for the cross section area at $x$. There is a wealth of literature on the quasi--one--dimensional compressible Euler equations. One can refer to \cite{LIu2, jiahui, XinZ, CGQ2024} and the references therein.
\par For simplicity, we consider ideal polytropic gases, for which the total energy $E$, the sonic speed $c$ and the Mach number $M$ are represented as
\begin{eqnarray*}
	E=\frac{u^{2}}{2}+\frac{p}{(\gamma-1)\rho}, \quad c=\sqrt{\frac{\gamma p}{\rho}}, \quad M=\frac{u}{c}=\sqrt{\frac{\rho u^{2}}{\gamma p}}.
\end{eqnarray*}
with $\gamma \in (1,3)$. We assume that the nozzles are divergent, i.e.,
\begin{align}
	0<m_{a}\le \frac{a^{\prime}(x)}{a(x)} \le  M_{a} < +\infty,\label{a2}
\end{align}
and the initial--boundary data for equations \eqref{a1} are
\begin{align}
	(\rho,u,p)(t,x)\big|_{t=0}&=(\rho_{0},u_{0},p_{0})(x),\label{a3}\\
	(\rho,u,p)(t,x)\big|_{x=0}&=(\rho_{b-},u_{b-},p_{b-})(t), \quad p(t,x)\big|_{x=L}=p_{b+}(t).\label{a5}
\end{align}
In engineering, the system is an aerodynamic model of aircraft nozzles. 
\par A transonic shock solution for problem \eqref{a1}, \eqref{a3}-\eqref{a5} is defined as follows.
\begin{definition}\label{D1}
	A piecewise $C^1$ smooth solution to~\eqref{a1} is said to be a {\rm{transonic shock solution}} if it is separated by a shock wave positioned at $x=\chi(t)$, mediating a transition from supersonic upstream flow (left of the shock) to subsonic downstream flow (right of the shock), i.e., is of the form
	\begin{align*}
		(\rho,u,p)(t,x)=
		\left\{\begin{aligned}
			&(\rho_{-},u_{-},p_{-})(t,x),\quad 0\leq x<\chi(t),\\
			&(\rho_{+},u_{+},p_{+})(t,x),\quad \chi(t)<x\leq L,
		\end{aligned}\right.
	\end{align*}
	satisfying the \textsl{Rankine--Hugoniot} conditions
	\begin{align}
		&\big[\rho u\big]=\chi^{\prime}(t)\big[\rho\big], \label{hr1}\\
		&\big[\rho u^{2} + p\big]=\chi^{\prime}(t)\big[\rho u\big], \label{hr2}\\
		&\big[\frac{1}{2}\rho u^{3} + \frac{\gamma}{\gamma -1}p u\big]=\chi^{\prime}(t)\big[\frac{1}{2}\rho u^{2} + \frac{1}{\gamma -1}p\big], \label{hr3}
	\end{align}
	where $[f] := f(t,\chi(t)+) - f(t,\chi(t)-)$ denotes the jump across the shock, and the Lax geometric entropy condition
	\begin{align}
		\left(u - \sqrt{\frac{\gamma p}{\rho}}\right)(t,\chi(t)-) &> \chi^{\prime}(t) > \left(u - \sqrt{\frac{\gamma p}{\rho}}\right)(t,\chi(t)+), \label{Lax1}\\
		\left(u + \sqrt{\frac{\gamma p}{\rho}}\right)(t,\chi(t)+) &> \chi^{\prime}(t). \label{Lax2}
	\end{align}
\end{definition}
\par Early research on transonic shocks can be traced back to the pioneering work of Courant and Friedrichs \cite{Courant} on De Laval nozzles in the 1940s. Their analysis reveals that shock formation at specific divergent nozzle sections requires two critical conditions: sufficient outlet pressure and sustained supersonic inflow. Subsequently, Liu $\emph{et al.}$~\cite{Liu1, Lic, Glaz} used the Lax--Glimm scheme for hyperbolic conservation laws to analyze shock stability. The significant progress in multidimensional analyses comes from Xin and Yin~\cite{Xin}, where they establish the global existence and stability for symmetric transonic shocks in two-- and three--dimensional nozzles under a controlled outlet pressure while demonstrating structural instability of shocks in converging nozzles. Rauch $\emph{et al.}$~\cite{Rauch} later weakened stability criteria for quasi--one--dimensional flows by developing exponential decay estimates in linearized systems, relaxing prior constraints on nozzle geometry and shock magnitude. We also cite related works on transonic shock solutions of Euler equations \cite{Chen,Duan,Liao,Yin,YuanH,FangB,XinZ}.
\par For steady solutions, the system \eqref{a1},\eqref{a3}-\eqref{a5} becomes 
\begin{align}\label{a6}
	\left\{\begin{aligned}
		&\frac{d}{dx}\left(\tilde{\rho}\tilde{u}\right)=-\frac{a^{\prime}(x)}{a(x)}\tilde{\rho}\tilde{u},\\
		&\frac{d}{dx}\left(\tilde{\rho}\tilde{u}^{2}+\tilde{p}\right)=-\frac{a^{\prime}(x)}{a(x)}\tilde{\rho}\tilde{u}^{2},\\
		&\frac{d}{dx}\left(\frac{1}{2}\tilde{\rho}\tilde{u}^{3}+\frac{\gamma}{\gamma-1}\tilde{p}\tilde{u}\right)
		=-\frac{a^{\prime}(x)}{a(x)}\left(\frac{1}{2}\tilde{\rho}\tilde{u}^{3}+\frac{\gamma}{\gamma-1}\tilde{p}\tilde{u}\right),
	\end{aligned}\right.
\end{align}
with the boundary data
\begin{align}
	(\tilde{\rho},\tilde{u},\tilde{p})(x)\big|_{x=0}=(\tilde{\rho}_{-}(0),\tilde{u}_{-}(0),\tilde{p}_{-}(0)),\quad \tilde{p}(x)\big|_{x=L}=\tilde{p}_{+}(L).\label{a7}
\end{align}
It has been analyzed in \cite{Liu1, Embid} that if the inlet state $(\tilde{\rho}_{-}(0),\tilde{u}_{-}(0),\tilde{p}_{-}(0))$ and outlet pressure $\tilde{p}_{+}(L)$ satisfy compatibility criteria
\begin{align}\label{compatibCri}
	\tilde{u}_{-}(0) > \tilde{c}_{-}(0) \mathop{=}\limits^{\triangle}\sqrt{\frac{\gamma \tilde{p}}{\tilde{\rho}}}(0), \quad \tilde{p}_{+}(L) \in I_{p} = \left(p_{\min}, p_{\max}\right),
\end{align}
where the interval $I_{p}$ depends on the inlet state $(\tilde{\rho}_{-}(0),\tilde{u}_{-}(0),\tilde{p}_{-}(0))$ and the geometric shape of the nozzle $a(x), L$, then there exits a unique steady transonic shock solution $(\tilde{\rho},\tilde{u},\tilde{p})(x)$ such that a shock at $x=\tilde{x}$ separates the nozzle into a supersonic upstream solution $(\tilde{\rho}_{-},\tilde{u}_{-},\tilde{p}_{-})(x)$ over $[0,\tilde{x})$ and a subsonic downstream solution $(\tilde{\rho}_{+},\tilde{u}_{+},\tilde{p}_{+})(x)$ over $(\tilde{x},L]$. 
\par Similar to~\cite{Rauch}, for some $\delta>0$, the supersonic branch admits prolongation to $[0,\tilde{x}+\delta]$ via local extension of~\eqref{a6} and \eqref{a7}, while the subsonic branch similarly extends to $[\tilde{x}-\delta,L]$, preserving original solution characteristics within their respective domains. Throughout this paper, the notation $(\tilde{\rho}_{\pm},\tilde{u}_{\pm},\tilde{p}_{\pm})(x)$ represents these extended solutions.
\par The dynamic behavior of transonic shocks is highly sensitive to perturbations from both upstream and downstream flows. Bruce and Babinsky \cite{Bruce} conducted an experimental study in 2008, adjusting the downstream flow to control the periodic movement of the shocks in a straight nozzle. However, a mathematical framework for the temporal evolution of transonic shocks under boundary perturbations is still lacking. In this paper, we conduct an analysis of the time--varying shock dynamics. Specifically, we consider the periodic boundary data of the form 
\begin{align}
\rho_{b-}(t)=\tilde{\rho}_{-}(0)+\bar{\rho}_{b-}(t),\quad &u_{b-}(t)=\tilde{u}_{-}(0)+\bar{u}_{b-}(t),\quad p_{b-}(t)=\tilde{p}_{-}(0)+\bar{p}_{b-}(t),\label{a8}\\
&p_{b+}(t)=\tilde{p}_{+}(L)+\bar{p}_{b+}(t),\label{a9}
\end{align}
which is a $C^{1}$ perturbation of \eqref{a7}
\begin{align}
	\|\bar{\rho}_{b-}(t)\|_{C^{1}(\mathbb{R}_{+})}+\|\bar{u}_{b-}(t)\|_{C^{1}(\mathbb{R}_{+})}+\|\bar{p}_{b-}(t)\|_{C^{1}(\mathbb{R}_{+})}+\|\bar{p}_{b+}(t)\|_{C^{1}(\mathbb{R}_{+})}\leq\epsilon, \label{a13}
\end{align}
for some small constant $\epsilon>0$. We further assume the boundary states have same constant temporal period $\mathcal{T}>0$, i.e.,
\begin{align}
\bar{\rho}_{b-}(t+\mathcal{T})=\bar{\rho}_{b-}(t),~~\bar{u}_{b-}(t+\mathcal{T})=\bar{u}_{b-}(t),
~~\bar{p}_{b-}(t+\mathcal{T})=\bar{p}_{b-}(t),~~ \bar{p}_{b+}(t+\mathcal{T})=\bar{p}_{b+}(t).\label{a10}
\end{align}
\par Periodic smooth flows driven by periodic boundary data have drawn a lot of attentions. Matsumura and Nishida \cite{Matsumura} investigated the existence of time--periodic solutions to a viscous gas equation driven by the piston periodic boundary. Luo \cite{Luo} extended this result to polytropic gases under some restrictive conditions on the piston velocity. More recently, Yuan \cite{Yuan} demonstrated the existence of supersonic time--periodic solutions to isentropic flows, and Qu~\cite{Qu} studied the existence and stability of hyperbolic systems with dissipative periodic boundaries. For extended discussions, one can refer to \cite{Yuw, QUPENG, Yuh, Zhang2, Qup1, Zhangx}.
\subsection{Main results} 
In this paper, we focus on the existence and stability of time--periodic transonic shock solutions to \eqref{a1}. We denote by $\tilde{B}$ the Bernoulli constant of the background flow $(\tilde{\rho},\tilde{u},\tilde{p})(x)$, defined as
\begin{eqnarray*}
	\tilde{B} := \dfrac{1}{2}\tilde{u}_{-}^{2}(0) + \dfrac{\gamma}{\gamma -1}\dfrac{\tilde{p}_{-}(0)}{\tilde{\rho}_{-}(0)},
\end{eqnarray*}
and denote by $\tilde{M}_{-}(0)$ the inlet Mach number, and then further assume that the inlet data \eqref{a7} satisfy
\begin{align}
&\frac{a^{\prime}(x)}{a(x)}\sqrt{\left(\gamma-1\right)\tilde{B}}<\mathcal{E} ,\label{a11}\\
&1<\tilde{M}_{-}(0)<M_{crit},\quad \text{if $\gamma \le \gamma_{c}$},\label{AA11}
\end{align}
for some constants $0<\mathcal{E}\ll 1$, $M_{crit} >1$. And $\gamma_{c} \in (1,\frac{7}{5})$ would be determined later in \Cref{lemmaDissipa}.
\par The main theorems of this paper are presented as follows. First, we will prove the the existence of temporal periodic transonic shock solutions.
\begin{theorem}[Existence of time-periodic transonic shock solutions]\label{t1}
With boundary data $\rho_{b-}, u_{b-}, p_{b-}$ and $p_{b+}$ as \eqref{a8}--\eqref{a9}, under the assumptions \eqref{compatibCri} and \eqref{a11}--\eqref{AA11}, there exist a small constant $\epsilon_1>0$ and a constant $C_{E} > 0$ such that
for any given $\epsilon\in(0,\epsilon_1)$, any $\mathcal{T}>0$, and any given boundary perturbations $\bar{\rho}_{b-},\bar{u}_{b-},\bar{p}_{b-},\bar{p}_{b+}\in C^1_{t}$ satisfying \eqref{a13}--\eqref{a10}, there exists an initial data $(\rho^{(\mathcal{T})}_{0},u^{(\mathcal{T})}_{0},p^{(\mathcal{T})}_{0})(x)$ of the form 
\begin{align}\label{a14}
	(\rho^{(\mathcal{T})}_{0},u^{(\mathcal{T})}_{0},p^{(\mathcal{T})}_{0})(x)=
	\left\{\begin{aligned}
		&(\rho_{0-}^{(\mathcal{T})},u_{0-}^{(\mathcal{T})},p_{0-}^{(\mathcal{T})})(x),\quad 0\leq x<\chi^{(\mathcal{T})}(0),\\
		&(\rho_{0+}^{(\mathcal{T})},u_{0+}^{(\mathcal{T})},p_{0+}^{(\mathcal{T})})(x),\quad \chi^{(\mathcal{T})}(0)<x\leq L,
	\end{aligned}\right.
\end{align}
satisfying
\begin{multline}
	|\chi^{(\mathcal{T})}(0)-\tilde{x}|+\|(\rho_{0-}^{(\mathcal{T})},u_{0-}^{(\mathcal{T})},p_{0-}^{(\mathcal{T})})
	-(\tilde{\rho}_{-},\tilde{u}_{-},\tilde{p}_{-})\|_{C^{1}([0,\chi^{(\mathcal{T})}(0)])}\\
	+\|(\rho_{0+}^{(\mathcal{T})},u_{0+}^{(\mathcal{T})},p_{0+}^{(\mathcal{T})})-(\tilde{\rho}_{+},\tilde{u}_{+},
	\tilde{p}_{+})\|_{C^{1}([\chi^{(\mathcal{T})}(0),L])}\leq C_{E}\epsilon,\label{a15}
\end{multline}
such that the initial--boundary value problem~\eqref{a1},\eqref{a3}--\eqref{a5} admits a transonic shock solution $(\rho^{(\mathcal{T})},u^{(\mathcal{T})},p^{(\mathcal{T})})(t,x)$ on $(t, x)\in \mathbb{R}_{+}\times[0,L]$ containing a shock $x=\chi^{(\mathcal{T})}(t) \in (0,L)$ with temporal periodicity
\begin{align}\label{temPeriod}
	(\rho^{(\mathcal{T})},u^{(\mathcal{T})},p^{(\mathcal{T})})(t+\mathcal{T},x) = (\rho^{(\mathcal{T})},u^{(\mathcal{T})},p^{(\mathcal{T})})(t,x), \quad \chi^{(\mathcal{T})}(t+\mathcal{T})=\chi^{(\mathcal{T})}(t).
\end{align}
\par Moreover, if the supersonic and subsonic regions are denoted by 
\begin{equation}\label{regions}
	\Omega_{-}^{\mathcal{T}} = \left\{(t,x):t\in \mathbb{R}_{+}, 0\le x \le\chi^{(\mathcal{T})}(t) \right\}, \quad \Omega_{+}^{\mathcal{T}} = \left\{(t,x):t\in \mathbb{R}_{+}, \chi^{(\mathcal{T})}(t)\le x \le L \right\},
\end{equation}
the transonic shock solution satisfies
\begin{multline}
\sum_{m=0}^{1}|\partial_{t}^{m}(\chi^{(\mathcal{T})}(t)-\tilde{x})|+\|(\rho_{-}^{(\mathcal{T})},u_{-}^{(\mathcal{T})},
p_{-}^{(\mathcal{T})})
-(\tilde{\rho}_{-},\tilde{u}_{-},\tilde{p}_{-})\|_{C^{1}(\Omega_{-}^{\mathcal{T}})}\\
+\|(\rho_{+}^{(\mathcal{T})},u_{+}^{(\mathcal{T})},p_{+}^{(\mathcal{T})})-(\tilde{\rho}_{+},\tilde{u}_{+},
\tilde{p}_{+})\|_{C^{1}(\Omega_{+}^{\mathcal{T}})}<C_{E}\epsilon,\label{a16}
\end{multline}
where $(\rho_{\pm}^{(\mathcal{T})},u_{\pm}^{(\mathcal{T})},p_{\pm}^{(\mathcal{T})})$ are the solution in associated regions. 
\end{theorem}
Second, we would like to investigate the dynamic stability of time--periodic transonic shock solutions under small initial perturbations.
\begin{theorem}[Stability of the time-periodic transonic shock solution]\label{t2}
With boundary data $\rho_{b-}, u_{b-}, p_{b-}$ and $p_{b+}$ as \eqref{a8}--\eqref{a9}, under the assumptions \eqref{compatibCri} and \eqref{a11}--\eqref{AA11}, there exists a constant $\epsilon_2\in(0,\epsilon_1)$ such that for any given $\epsilon\in(0,\epsilon_2)$, any $\mathcal{T}>0$, any given boundary perturbations  $\bar{\rho}_{b-},\bar{u}_{b-},\bar{p}_{b-},\bar{p}_{b+}\in C^1$ satisfying \eqref{a13}--\eqref{a10}, and any given initial data
\begin{align}\label{a19}
(\rho_{0},u_{0},p_{0})(x)=
\left\{\begin{aligned}
&(\rho_{0-},u_{0-},p_{0-})(x),\quad 0\leq x<x_{0},\\
&(\rho_{0+},u_{0+},p_{0+})(x),\quad x_{0}<x\leq L,
\end{aligned}\right.
\end{align}
with
\begin{multline}
|x_{0}-\tilde{x}|+\|(\rho_{0-},u_{0-},p_{0-})-(\tilde{\rho}_{-},\tilde{u}_{-},\tilde{p}_{-})\|_{C^{1}([0,x_0])}\\
+\|(\rho_{0+},u_{0+},p_{0+})-(\tilde{\rho}_{+},\tilde{u}_{+},\tilde{p}_{+})\|_{C^{1}([x_0,L])}\leq\epsilon,\label{a20}
\end{multline}
the initial--boundary value problem \eqref{a1},\eqref{a3}--\eqref{a5} admits a transonic shock solution, which includes the shock wave $x=\chi(t)$ with $\chi(0)=x_0,~ \chi(t)\in(0,L)$. 
\par Moreover, this transonic shock solution exponentially converges to the time--periodic solution driving by the same boundary condition, i.e., there exist constants $C_S>0$, $\zeta\in(0,1)$ and $\mathcal{T}_{0}>0$, such that for all $t>0$,
\begin{align}
&\|(\rho_{-},u_{-},p_{-})(t,\cdot)-(\rho_{-}^{(\mathcal{T})},u_{-}^{(\mathcal{T})},p_{-}^{(\mathcal{T})})(t,\cdot)\|
_{C^{0}([0,\min\{\chi(t),\chi^{(\mathcal{T})}(t)\}))}\notag\\
&+\|(\rho_{+},u_{+},p_{+})(t,\cdot)-(\rho_{+}^{(\mathcal{T})},u_{+}^{(\mathcal{T})},p_{+}^{(\mathcal{T})})(t,\cdot)\|
_{C^{0}((\max\{\chi(t),\chi^{(\mathcal{T})}(t)\},L])}\notag\\
&+\Big|\chi(t)-\chi^{(\mathcal{T})}(t)\Big|+\Big|\partial_{t}\chi(t)-\partial_{t}\chi^{(\mathcal{T})}(t)\Big|\leq C_{S}\epsilon\zeta^{\lfloor\frac{t}{\mathcal{T}_{0}}\rfloor},\label{a21}
\end{align}
where $(\rho_{\pm},u_{\pm},p_{\pm})$ and $(\rho_{\pm}^{(\mathcal{T})},u_{\pm}^{(\mathcal{T})},p_{\pm}^{(\mathcal{T})})$ are the solutions in associated regions with initial data \eqref{a19}--\eqref{a20} and \eqref{a14}--\eqref{a15}, respectively.
\end{theorem}
With \Cref{t2}, by taking $t \to +\infty$ in \eqref{a21}, one can directly get the follow uniqueness result for the temporal periodic transonic solution.
\begin{corollary}
	With boundary data $\rho_{b-}, u_{b-}, p_{b-}$ and $p_{b+}$ as \eqref{a8}--\eqref{a9}, under the assumptions \eqref{compatibCri} and \eqref{a11}--\eqref{AA11}, there exists a constant $\epsilon_3\in(0,\epsilon_2)$ such that for any given $\epsilon\in(0,\epsilon_3)$, any $\mathcal{T}>0$, any given boundary perturbations  $\bar{\rho}_{b-},\bar{u}_{b-},\bar{p}_{b-},\bar{p}_{b+}\in C^1$ satisfying \eqref{a13}--\eqref{a10}, the corresponding time--periodic solution $(\rho^{(\mathcal{T})},u^{(\mathcal{T})},p^{(\mathcal{T})})$ containing corresponding shock $\chi^{(\mathcal{T})}$ is unique.
\end{corollary}
We provide several remarks regarding the \Cref{t1,t2}.
\begin{remark}
	From the derivation process in \Cref{s3} of this paper, it can be seen that in fact, \eqref{a16} and \eqref{a21} can include the estimates for $|\partial_{t}^{2}\chi|$ and $|\partial_{t}^{2}\chi - \partial_{t}^{2}\chi^{(\mathcal{T})}|$ respectively. Moreover, analogous to \cite{Qu}, one can get the conclusion of regularity for the time--periodic solution as well as stabilization around it.
\end{remark}
\begin{remark}
	In the prior study \cite{ZHANG}, the smallness on the expansion rate of the nozzles is required, i.e., $\frac{a^{\prime}(x)}{a(x)}\le \kappa \ll 1$. Here, we relax the restriction to some extent and use a more essential assumption \eqref{a11}.
\end{remark}
\begin{remark}
	Still in the prior study \cite{ZHANG} for the isothermal gases, the inlet Mach number is limited to $\tilde{M}_{-}(0) < \sqrt{2}+1$. Here, we generalize this condition to all $\gamma \in (1,3)$, and prove that for larger gammas (especially for the air $\gamma=1.4$), no restrictions are required. Moreover, we believe the restrain \eqref{AA11} is necessary in a sense as it describes the instability of high temperature or hypersonic shocks.
\end{remark}
\subsection{A sketch of this paper}
\par For the full compressible Euler system that we are concerning, it is well--known that it has three characteristics, namely the left and right acoustic waves, and the entropy wave that moves along with the material flow. In the supersonic region, all the three waves have positive propagation speeds. Thus, the inlet boundary perturbations only propagate downstream without causing any resonance. However, across a transonic shock into the subsonic region, although the material still flows downstream, the left acoustic wave propagates backward and causes resonance at the shock. That is, stronger wave perturbations lead to stronger perturbations in shock positions, which in turn enhance the wave perturbations. In addition, there is a fundamental difficulty that the problem involves a free boundary on the left side of the subsonic region.
\par The key to solving these difficulties lies in the well--known \textsl{Rankine--Hugoniot} conditions (R--H conditions for short). By the implicit function theorem, we determine the entropy wave and the right acoustic wave emerging from the shock wave as functions of the left acoustic wave and other perturbations. At the same time, we deduce the ordinary differential equation for the shock wave position as a function of time. These help to rewrite the system in the form of wave decomposition (\cref{waveDe}).
\par The central part of this paper---and also the mechanism that ensures the theorems can be proved---is the analysis of the structural dissipation of the shock (\Cref{sectionDissi}). We show that taking into account the resonance caused by the change in shock position, the shock boundary still has a dissipative effect on the reflected waves in the subsonic region. This dissipation is entirely determined by the R--H conditions and is a physical mechanism resulting from the nature of shock waves. To the best of our knowledge, this is the first comprehensive analysis of $C^{1}$ shock structural dissipation for general ideal gases.
\par For the proof of \Cref{t1}, we design a fraction--step iteration scheme. It is basically a linearized iteration of the inhomogeneous terms and boundary conditions, plus a fractional iteration step involving the shock position.
\par It is interesting to see that the expansion of the nozzles (assumption \eqref{a2}) and the positive background velocity help to stabilize the shock position. However, they will increase the inhomogeneous term, and thus enhance the resonance. That is why we need the restriction \eqref{a11}. The stability of periodic transonic flows in a general nozzle remains open.
\par This paper is organized as follows: \Cref{s2} gives some useful preliminaries, including the existence of steady transonic shock solutions, the lemmas on the existence of periodic solutions for a kind of ODE, the change of variables, and the analysis of time--periodic solutions in the supersonic region. \Cref{sectionDissi} investigates the shock boundary condition and its structural dissipation as well as reformulates the problem in the subsonic region. \Cref{s3,s4} are subsequently devoted to providing rigorous proof of \Cref{t1} and \Cref{t2} respectively.

\section{Preliminaries}\label{s2}
In this section, we give four parts of preliminaries.
\subsection{Well--posedness of steady transonic shock solutions}\label{su1}
In this subsection, We sketch the proof of the existence and uniqueness for the steady transonic shock solutions. One can get more details in \cite{Chen,Glaz} as well as \cite{Duan}. 
\par For any $x_{in} \in [0, L]$, consider the solution governed by the system of ODEs \eqref{a6} with inlet data
\begin{align}\label{inlet}
	(\tilde{\rho},\tilde{u},\tilde{p})(x)\big|_{x=x_{in}}=(\tilde{\rho}_{in},\tilde{u}_{in},\tilde{p}_{in}).
\end{align}
For divergent nozzles with \eqref{a2}, the standard theory for ODEs guarantees the existence and  uniqueness of a classical solution on $x \in [x_{in}, L]$ if the inlet data subject to
$$\tilde{\rho}_{in}>0, \tilde{u}_{in}>0,\tilde{p}_{in}>0,  \tilde{u}_{in}\neq \tilde{c}_{in}:=\sqrt{\frac{\gamma \tilde{p}_{in}}{\tilde{\rho}_{in}}}.$$
\par Moreover, direct derivation from \eqref{a6} yields
\begin{equation}\label{b2}
	\left \{
	\begin{aligned}
		&\dfrac{d\tilde{\rho}}{dx} = -\dfrac{a^{\prime}(x)}{a(x)}\dfrac{\tilde{\rho}^{2}\tilde{u}^{2}}{\tilde{\rho}\tilde{u}^{2}-\gamma\tilde{p}} = -\dfrac{a^{\prime}(x)}{a(x)}\dfrac{\tilde{\rho}\tilde{u}^{2}}{\tilde{u}^{2}-\tilde{c}^{2}},  \\
		&\dfrac{d\tilde{u}}{dx} = \dfrac{a^{\prime}(x)}{a(x)}\dfrac{\gamma\tilde{p}\tilde{u}}{\tilde{\rho}\tilde{u}^{2}-\gamma\tilde{p}} = \dfrac{a^{\prime}(x)}{a(x)}\dfrac{\tilde{c}^{2}\tilde{u}}{\tilde{u}^{2}-\tilde{c}^{2}},	\\
		&\dfrac{d\tilde{p}}{dx} = -\dfrac{a^{\prime}(x)}{a(x)}\dfrac{\gamma\tilde{p}\tilde{\rho}\tilde{u}^{2}}{\tilde{\rho}\tilde{u}^{2}-\gamma\tilde{p}} = -\dfrac{a^{\prime}(x)}{a(x)}\dfrac{\tilde{c}^{2}\tilde{\rho}\tilde{u}^{2}}{\tilde{u}^{2}-\tilde{c}^{2}}.
	\end{aligned}
	\right . 
\end{equation}
Thus, one can easily derive (see for example \cite{Ma}) that the solution satisfies monotonicity
\begin{align}
&\dfrac{d\tilde{u}}{dx}>0,~~\dfrac{d\tilde{c}}{dx}<0,\quad\quad \text{if}~ \tilde{u}_{in}>\tilde{c}_{in}, \label{b3} \\
&\dfrac{d\tilde{u}}{dx}<0,~~\dfrac{d\tilde{c}}{dx}>0,\quad\quad \text{if}~0<\tilde{u}_{in}<\tilde{c}_{in}. \label{b4}
\end{align}
\par Next, we consider a steady transonic shock solution, as defined in \Cref{D1}, with a steady shock at $x = \tilde{x}$. We construct a formal solution as
\begin{align}\label{steadySolu}
	(\tilde{\rho},\tilde{u},\tilde{p})(x)=
	\left\{\begin{aligned}
		&(\tilde{\rho}_{-},\tilde{u}_{-},\tilde{p}_{-})(x),&&0\leq x<\tilde{x},\\
		&(\tilde{\rho}_{+},\tilde{u}_{+},\tilde{p}_{+})(x;\tilde{x}),&&\tilde{x}<x\leq L.
	\end{aligned}\right.
\end{align}
The upstream flow $(\tilde{\rho}_{-},\tilde{u}_{-},\tilde{p}_{-})(x)$ is defined as the solution to ODE system \eqref{a6} and \eqref{inlet} with 
\begin{eqnarray*}
	x_{in}=0, \quad (\tilde{\rho}_{in},\tilde{u}_{in},\tilde{p}_{in}) = (\tilde{\rho}_{-}(0),\tilde{u}_{-}(0),\tilde{p}_{-}(0)).
\end{eqnarray*}
Across the shock, we set the right states being the ones derived from the R--H conditions \eqref{hr1}--\eqref{hr3} with $\chi^{\prime}(t)\equiv 0$ as
\begin{align}
	\nonumber
	\tilde{\rho}_{r}=\frac{(\gamma+1)\tilde{\rho}_{l}^{2}\tilde{u}_{l}^{2}}{2\gamma \tilde{p}_{l}+(\gamma-1)\tilde{\rho}_{l}\tilde{u}_{l}^{2}}, \qquad \tilde{u}_{r}=\frac{\gamma-1}{\gamma+1}\tilde{u}_{l}+\frac{2\gamma}{\gamma+1}\frac{\tilde{p}_{l}}{\tilde{\rho}_{l}\tilde{u}_{l}},
	\qquad \tilde{p}_{r}=\frac{2\tilde{\rho}_{l}\tilde{u}_{l}^{2}-(\gamma-1)\tilde{p}_{l}}{\gamma+1},
\end{align}
where $(\tilde{\rho}_{l},\tilde{u}_{l},\tilde{p}_{l}):=(\tilde{\rho}_-, \tilde{u}_-, \tilde{p}_-)(\tilde{x}-)$ denoting the states on the left of the shock. We define the downstream flow $(\tilde{\rho}_{+},\tilde{u}_{+},\tilde{p}_{+})(x;\tilde{x})$ as the solution to ODE system \eqref{a6} and \eqref{inlet} with 
\begin{eqnarray*}
	x_{in}=\tilde{x}, \quad (\tilde{\rho}_{in},\tilde{u}_{in},\tilde{p}_{in}) = (\tilde{\rho}_{r},\tilde{u}_{r},\tilde{p}_{r}).
\end{eqnarray*}
Along with the first criterion of \eqref{compatibCri} and monotonicity \eqref{b3}--\eqref{b4}, one can see the upstream flow is supersonic, while the downstream one is subsonic. Thus, it meets the Lax entropy condition \eqref{Lax1}--\eqref{Lax2}. Therefore, the solution given by \eqref{steadySolu} is a well--defined transonic shock solution to \eqref{a6}--\eqref{a7} if $\tilde{p}_{+}(L;\tilde{x})$ coincides with $\tilde{p}_{+}(L)$.
\par In \cite[Lemma 3.1]{Duan}, the monotonicity and continuity of $\tilde{p}_{+}(L;\tilde{x})$ with respect to $\tilde{x}$ have been proved. That is, for a divergent nozzle ($a^{\prime}(x)>0$), 
\begin{eqnarray*}
	\dfrac{d\tilde{p}_{+}(L;\tilde{x})}{d\tilde{x}}<0.
\end{eqnarray*}
Now, we set
\begin{eqnarray*}
	p_{\min} = \tilde{p}_{+}(L;L), \quad p_{\max} = \tilde{p}_{+}(L;0),
\end{eqnarray*}
then one can see for any $\tilde{p}_{+}(L) \in I_{p} = (p_{\min}, p_{\max})$, there exists a unique $\tilde{x} \in (0,L)$ such that the steady transonic shock solution given by \eqref{steadySolu} satisfies $\tilde{p}_{+}(L;\tilde{x}) = \tilde{p}_{+}(L)$, and is thus exactly the unique steady transonic shock solution to \eqref{a6}--\eqref{a7}.
\par Throughout the paper, for fixed boundary data \eqref{a7}, we take the steady transonic shock solution of \eqref{a6}--\eqref{a7} as the background solution. Furthermore, by using the corresponding ODEs,  one can successively derive the above and below boundedness of Mach number $M$, sonic speed $c$, and $\rho$, $u$, $p$, namely, 
\begin{eqnarray*}
	&M_{\min} \le \tilde{M} \le M_{\max}, \quad c_{\min} \le \tilde{c} \le c_{\max}, \\ &u_{\min} \le \tilde{u} \le u_{\max}, \quad
	\rho_{\min} \le \tilde{\rho} \le \rho_{\max}, \quad p_{\min} \le \tilde{p} \le p_{\max},
\end{eqnarray*}
where the constant bounds only depend on $m_{a}$, $M_{a}$, $L$ and inlet data \eqref{a7}.

\subsection{Periodic solutions to ODE systems driven by periodic perturbations}\label{su11}
In this subsection, we consider the ODE systems of the follow form, which describe the dynamics of the shock position under time--periodic perturbations
\begin{align}
	\frac{d\phi(t)}{dt}=\Psi(\phi(t),\varsigma_{1}(t,\phi(t)),\varsigma_{2}(t,\phi(t)),\varsigma_{3}(t,\phi(t)),\varsigma_{4}(t,\phi(t)),
	\varsigma_{5}(t,\phi(t))),\label{bb1}
\end{align}
where $\varsigma_{i}(t,x)~(i=1,2,\ldots,5)$ are $C^{1}$ perturbation functions on $\mathbb{R}\times[0,L]$ satisfying
\begin{align}
	\varsigma_{i}(0,0)=0, \quad \varsigma_{i}(t+\mathcal{T},x)=\varsigma_{i}(t,x),\quad i=1,2,\ldots,5.\label{TB1}
\end{align} 
Moreover, we assume $\Psi(\phi,\varsigma_{1},\varsigma_{2},\varsigma_{3},\varsigma_{4},\varsigma_{5})$ is a smooth function satisfying
\begin{align}
	&\Psi(0,0,0,0,0,0)=0,\label{bb2}\\
	&\frac{\partial\Psi}{\partial\phi}(0,0,0,0,0,0)<0.\label{bb3}
\end{align}
For simplicity, we denote $\Psi(\phi,\varsigma_{1},\varsigma_{2},\varsigma_{3},\varsigma_{4},\varsigma_{5})$ by $\Psi(\phi, \bm{\varsigma})$ and
$$\Psi_{\varsigma_{i,0}}=\Big|\frac{\partial\Psi}{\partial\varsigma_{i}}(0,\bm{0})\Big|,~(i=1,2,\ldots,5),
\quad\Psi_{\phi,0}=\Big|\frac{\partial\Psi}{\partial\phi}(0,\bm{0})\Big|.$$
Under these assumptions, the following two lemmas hold. The proof of them is not the main point of this paper, so we omit it. One can find it in the appendix of~\cite{ZHANG}. Here and throughout the paper, we denote the $C^{0}$ norm by $\|\cdot\|$ for simplification.
\begin{lemma}[Existence and uniqueness of the periodic solution]\label{lemmaODE1}
	There exist constants $\varepsilon_{1}>0$ and $C_\Psi>0$ such that for any perturbations satisfying $\|\varsigma_{i}\|_{C^{1}}<\varepsilon_{1}(i=1,2,\ldots,5)$, system~\eqref{bb1} admits a unique periodic solution $\phi^{*}(t)$ with $\phi^{*}(t+\mathcal{T})=\phi^{*}(t)$.
	Furthermore, the following estimates hold
	\begin{align}
		\|\phi^{*}\|\leq(1+C_{\Psi}\varepsilon_{1})\frac{1}{\Psi_{\phi,0}}\sum_{i=1}^{5}\Psi_{\varsigma_{i,0}}\|\varsigma_{i}\|,\label{bb19}
	\end{align}
	\begin{align}
		\|{\phi^{*}}^{\prime}\|\leq
		(2+C_{\Psi}\varepsilon_{1})\sum_{i=1}^{5}\Psi_{\varsigma_{i,0}}\|\varsigma_{i}\|,\label{bb20}
	\end{align}
	\begin{align}
		\|{\phi^{*}}^{\prime\prime}\|\leq&(2+C_{\Psi}\varepsilon_{1})\Psi_{\phi,0}\sum_{i=1}^{5}\Psi_{\varsigma_{i,0}}\|\varsigma_{i}\|
		+(1+C_{\Psi}\varepsilon_{1})\sum_{i=1}^{5}\Psi_{\varsigma_{i,0}}
		\|\frac{\partial\varsigma_{i}}{\partial t}\|.\label{bb21}
		\end{align}
\end{lemma}
\begin{lemma}[$C^{1}$ stability estimates with respect to perturbations]\label{lemmaODE2}
	Let be $\phi_{1}^{*}(t)$ and $\phi_{2}^{*}(t)$ be two periodic solutions driven by different group of perturbation functions ${\varsigma_{1}}_{i}(t,\phi)(i=1,2,\ldots,5)$ and ${\varsigma_{2}}_{i}(t,\phi)(i=1,2,\ldots,5)$ with same periodicity, i.e.,
	\begin{align}
		\frac{d\phi_{j}}{dt}(t)=\Psi(\phi_{j}(t), \bm{\varsigma_{j}}(t,\phi_{j}(t))), \quad j=1,2,\label{bb29}
	\end{align}
	and
	\begin{align*}
		\phi_{1}^{*}(t+\mathcal{T})=\phi_{1}^{*}(t),~~\phi_{2}^{*}(t+\mathcal{T})=\phi_{2}^{*}(t),\quad \forall t\in\mathbb{R}.
	\end{align*}
	Then there exist constants $\varepsilon_{2}>0$ such that for any perturbations satisfying $\|{\varsigma_{1}}_{i}\|_{C^{1}}\leq\varepsilon_{2},~\|{\varsigma_{2}}_{i}\|_{C^{1}}\leq\varepsilon_{2},$
	the following estimates hold
	\begin{align}
		\|\phi_{1}^{*}-\phi_{2}^{*}\|\leq(1+C_{\Psi}\mathcal{T}\varepsilon_{2})\exp(\Psi_{\phi,0}\mathcal{T})\frac{1}{\Psi_{\phi,0}}
		\sum_{i=1}^{5}\Psi_{\varsigma_{i,0}}\|{\varsigma_{1}}_{i}-{\varsigma_{2}}_{i}\|,
		\label{bb33}
	\end{align}
	\begin{align}
		\|{{\phi}_{1}^{*}}^{\prime}-{{\phi}_{2}^{*}}^{\prime}\|
		\leq&(1+C_{\Psi}\mathcal{T}\varepsilon_{2})\exp(\Psi_{\phi,0}\mathcal{T})\sum_{i=1}^{5}\Psi_{\varsigma_{i,0}}\|{\varsigma_{1}}_{i}
		-{\varsigma_{2}}_{i}\|
		+(1+C_{\Psi}\varepsilon_{2})\sum_{i=1}^{5}\Psi_{\varsigma_{i,0}}\|{\varsigma_{1}}_{i}-{\varsigma_{2}}_{i}\|.
		\label{bb34}
	\end{align}
\end{lemma}

\subsection{Variable substitution and the wave decomposition}\label{su2}
Now, we start with some basic analysis of system \eqref{a1} and carry out a substitution into a set of approximate diagonalization variables.
\par Denote the unknowns by $U=\left(\rho, u, p\right)^{\top}$, the fixed background solution by $(\tilde{\rho}, \tilde{u}, \tilde{p})(x)$, the background shock position by $\tilde{x}$, and the perturbations by
\begin{eqnarray*}
	\bar{\rho} = \rho - \tilde{\rho}, \quad \bar{u} = u - \tilde{u}, \quad \bar{p} = p - \tilde{p}.
\end{eqnarray*} 
We choose a set of variables $\Upsilon (x,U) = (\Upsilon_{1}, \Upsilon_{2}, \Upsilon_{3})^{\top}$ given by
\begin{equation}\label{newVaria}
	\Upsilon_{1} = 2\sqrt{p} - \sqrt{\gamma \rho}\left(u - \tilde{u}\right), \quad \Upsilon_{2} = \ln p - \gamma \ln \rho, \quad
	\Upsilon_{3} = 2\sqrt{p} + \sqrt{\gamma \rho}\left(u - \tilde{u}\right).
\end{equation}
For the background solution, we denote
\begin{equation}
	\tilde{\Upsilon}_{1} = 2\sqrt{\tilde{p}}, \quad
	\tilde{\Upsilon}_{2} = \ln \tilde{p} - \gamma \ln \tilde{\rho}, \quad
	\tilde{\Upsilon}_{3} = 2\sqrt{\tilde{p}}.
\end{equation}
Accordingly, we denote the perturbation variables by $\bar{\Upsilon}(x,U) = \left(\bar{\Upsilon}_{1}, \bar{\Upsilon}_{2}, \bar{\Upsilon}_{3}\right)^{\top}$ as
\begin{equation}\label{upPertur}
	\bar{\Upsilon}_{1} = \Upsilon_{1} - \tilde{\Upsilon}_{1}, \quad \bar{\Upsilon}_{2} = \Upsilon_{2} - \tilde{\Upsilon}_{2}, \quad \bar{\Upsilon}_{3} = \Upsilon_{3} - \tilde{\Upsilon}_{3}.
\end{equation}
\par In order to carry out the variable substitution, we check the Jacobian
\begin{equation}\label{Jacobi}
	\nabla \bar{\Upsilon}\left(x,U\right) := \dfrac{\partial\left(\bar{\Upsilon}_{1}, \bar{\Upsilon}_{2}, \bar{\Upsilon}_{3}\right)}{\partial\left(\rho, u, p\right)} = 
	\begin{pmatrix}
		-\dfrac{1}{2}\sqrt{\dfrac{\gamma}{\rho}}\left(u-\tilde{u}\right) & -\sqrt{\gamma \rho} & \dfrac{1}{\sqrt{p}} \\[2ex]
		
		-\dfrac{\gamma}{\rho} & 0 & \dfrac{1}{p} \\[2ex]
		
		\dfrac{1}{2}\sqrt{\dfrac{\gamma}{\rho}}\left(u-\tilde{u}\right) & \sqrt{\gamma \rho} & \dfrac{1}{\sqrt{p}}
	\end{pmatrix}
	,
\end{equation}
and its inverse matrix
\begin{equation}\label{inverVariable}
	\left(\nabla \bar{\Upsilon}\left(x,U\right)\right)^{-1} =\dfrac{\partial\left(\rho, u, p\right)}{\partial\left(\bar{\Upsilon}_{1}, \bar{\Upsilon}_{2}, \bar{\Upsilon}_{3}\right)} =
	\begin{pmatrix}
		\dfrac{\rho}{2\gamma \sqrt{p}} & - \dfrac{\rho}{\gamma} & \dfrac{\rho}{2\gamma \sqrt{p}} \\[2ex]
		-\dfrac{1}{2\sqrt{\gamma \rho}}-\dfrac{u-\tilde{u}}{4\gamma\sqrt{p}} & \dfrac{u-\tilde{u}}{2\gamma} & \dfrac{1}{2\sqrt{\gamma \rho}}-\dfrac{u-\tilde{u}}{4\gamma\sqrt{p}} \\[2ex]
		\dfrac{\sqrt{p}}{2} & 0 & \dfrac{\sqrt{p}}{2}
	\end{pmatrix}
	.
\end{equation}
Therefore, under the assumption on smallness of perturbation, one can use the variable substitution $U = U\left(x,\bar{\Upsilon}\right)$. 
\par Next, we turn to the equivalent system under new variables $\bar{\Upsilon}$. For $C^{1}$ solutions, the system \eqref{a1} is equivalent to
\begin{equation}\label{dhyper}
	U_{t} + A(U)U_{x} = f\left(x,U\right),
\end{equation}
where
\begin{equation}\label{coeff}
	A(U) = 
	\begin{pmatrix}
		u & \rho & 0 \\
		0 & u & \rho^{-1} \\
		0 & \gamma p & u
	\end{pmatrix}
	\qquad f(x,U) = -\frac{a^{\prime}(x)}{a(x)}\left(\rho u, 0, \gamma p u\right)^{\top}.
\end{equation}
\par Applying the variable substitution \eqref{newVaria}--\eqref{upPertur}, one shall get the equivalent system
\begin{equation}\label{equivSys}
	\bar{\Upsilon}_{t} + \bar{A}\left(x, U(\bar{\Upsilon})\right)\bar{\Upsilon}_{x} = \bar{f}\left(x, U(\bar{\Upsilon})\right),
\end{equation}
where 
\begin{equation}\label{inhoTerm}
	\begin{split}
		\bar{A}\left(x, U(\bar{\Upsilon})\right) = \nabla \bar{\Upsilon}\left(x,U\right)\cdot A(U) \cdot \left(\nabla \bar{\Upsilon}\left(x,U\right)\right)^{-1},\\
		\bar{f}\left(x, U(\bar{\Upsilon})\right) = \nabla \bar{\Upsilon}\left(x,U\right) f(x,U) + \bar{A}\left(x, U(\bar{\Upsilon})\right)\dfrac{\partial}{\partial x}\bar{\Upsilon}\left(x,U\right).
	\end{split}
\end{equation}
From \eqref{Jacobi}, \eqref{inverVariable} and \eqref{coeff}, one has
\begin{equation}\label{transportCo}
	\bar{A}\left(x, U\right) = 
	\begin{pmatrix}
		u - c + \dfrac{\left(u-\tilde{u}\right)^{2}}{8c} & \sqrt{p}\left(\dfrac{u-\tilde{u}}{2}-\dfrac{\left(u-\tilde{u}\right)^{2}}{4c}\right) & -\dfrac{u-\tilde{u}}{2} + \dfrac{\left(u-\tilde{u}\right)^{2}}{8c} \\[2ex]
		0 & u & 0 \\[2ex]
		-\dfrac{u-\tilde{u}}{2} - \dfrac{\left(u-\tilde{u}\right)^{2}}{8c} & \sqrt{p}\left(\dfrac{u-\tilde{u}}{2}+\dfrac{\left(u-\tilde{u}\right)^{2}}{4c}\right) & u + c - \dfrac{\left(u-\tilde{u}\right)^{2}}{8c}	
	\end{pmatrix}
	.
\end{equation}
The system \eqref{equivSys} is strictly hyperbolic with three distinct eigenvalues
\begin{equation}
	\lambda_{1} = u - c, \quad \lambda_{2} = u, \quad \lambda_{3} = u + c.
\end{equation}
and one can choose the corresponding left eigenvectors
\begin{equation}\label{eigenVector}
	\begin{split}
		l_{1} &= \left(1, -2\sqrt{p}\dfrac{u- \tilde{u}}{4c + u - \tilde{u}}, \dfrac{u- \tilde{u}}{4c + u - \tilde{u}}\right) =: \left(1, l_{12}, l_{13}\right), \\
		l_{2} &= \left(0,1,0\right), \\
		l_{3} &= \left(\dfrac{u- \tilde{u}}{u - \tilde{u} - 4c}, -2\sqrt{p}\dfrac{u- \tilde{u}}{u - \tilde{u} -4c}, 1\right) =: \left(l_{31}, l_{32}, 1\right).
	\end{split}
\end{equation}
\par To calculate the inhomogeneous term, one can use the ODE system for the background solution \eqref{b2}, combining \eqref{Jacobi}, \eqref{inhoTerm} and \eqref{transportCo}, to get
\begin{equation}\label{source}
	\bar{f}\left(x, U(\bar{\Upsilon})\right) = -\frac{a^{\prime}(x)}{a(x)}\cdot\left(f_{1,L}+ f_{1,NL}, 0 ,f_{3,L} +f_{3,NL}\right)^{\top},
\end{equation}
where 
\begin{equation}
	\nonumber
	\begin{split}
		f_{1,L} = \gamma\sqrt{p}u -\dfrac{u-c}{\tilde{u} -\tilde{c}}\left(\gamma\sqrt{\tilde{p}}\tilde{u}\right)+ \dfrac{\gamma\sqrt{\tilde{p}}\tilde{u}}{2\left(\tilde{u}+\tilde{c}\right)}\left(u - \tilde{u}\right) - \dfrac{\sqrt{\gamma\rho}u}{2}\left(u-\tilde{u}\right) - \dfrac{u-c}{\tilde{u}^{2} - \tilde{c}^{2}}\sqrt{\gamma}\tilde{c}^{2}\tilde{u}\left(\sqrt{\rho} - \sqrt{\tilde{\rho}}\right), \\
		f_{3,L} = \gamma\sqrt{p}u -\dfrac{u+c}{\tilde{u} +\tilde{c}}\left(\gamma\sqrt{\tilde{p}}\tilde{u}\right)+ \dfrac{\gamma\sqrt{\tilde{p}}\tilde{u}}{2\left(\tilde{u}-\tilde{c}\right)}\left(u - \tilde{u}\right) + \dfrac{\sqrt{\gamma\rho}u}{2}\left(u-\tilde{u}\right) + \dfrac{u+c}{\tilde{u}^{2} - \tilde{c}^{2}}\sqrt{\gamma}\tilde{c}^{2}\tilde{u}\left(\sqrt{\rho} - \sqrt{\tilde{\rho}}\right), 
	\end{split}
\end{equation}
and $f_{1,NL}, f_{3,NL}$ are nonlinear higher--order terms with respect to $\rho- \tilde{\rho},  u-\tilde{u}, p -\tilde{p}$. Furthermore, we focus on the linear coefficients of the the inhomogeneous term with respect to $\bar{\Upsilon}$
\begin{equation}
	\nabla_{\bar{\Upsilon}}\bar{f}\left(x, U(\bar{\Upsilon})\right)\big|_{\bar{\Upsilon}=0} = \nabla_{U}\bar{f}\big|_{U=\tilde{U}}\cdot\dfrac{\partial U}{\partial \bar{\Upsilon}}\big|_{\bar{\Upsilon}=0}.
\end{equation}
From \eqref{inverVariable} and \eqref{source}, one has
\begin{equation}\label{linearPart}
	\nabla_{\bar{\Upsilon}}\bar{f}\left(x, U(0)\right) = -\frac{a^{\prime}(x)}{a(x)}
	\begin{pmatrix}
		\dfrac{\left(\gamma+1\right)\tilde{u}^{2}}{4\left(\tilde{u}-\tilde{c}\right)} + \dfrac{\tilde{c}^{2}}{2\left(\tilde{u}+\tilde{c}\right)} - \tilde{c} & \sqrt{\tilde{p}}\dfrac{\tilde{u}^{2}\tilde{c}}{\tilde{u}^{2}-
			\tilde{c}^{2}} & \dfrac{\left(\gamma-1\right)\tilde{u}^{2}}{4\left(\tilde{u}-\tilde{c}\right)} - \dfrac{\tilde{c}^{2}}{2\left(\tilde{u}-\tilde{c}\right)} \\[2ex]
		0 & 0 & 0 \\[2ex]
		\dfrac{\left(\gamma-1\right)\tilde{u}^{2}}{4\left(\tilde{u}+\tilde{c}\right)} - \dfrac{\tilde{c}^{2}}{2\left(\tilde{u}+\tilde{c}\right)} & -\sqrt{\tilde{p}}\dfrac{\tilde{u}^{2}\tilde{c}}{\tilde{u}^{2}-
			\tilde{c}^{2}} & \dfrac{\left(\gamma+1\right)\tilde{u}^{2}}{4\left(\tilde{u}+\tilde{c}\right)} + \dfrac{\tilde{c}^{2}}{2\left(\tilde{u}-\tilde{c}\right)} + \tilde{c}
	\end{pmatrix}
	.
\end{equation}
At this point, one can see that if $\tilde{u}-\tilde{c}$ is relatively small, then
\begin{equation}
	\nonumber
	\left(\nabla_{\bar{\Upsilon}}\bar{f}\left(x, U(0)\right)\right)_{11} > 0, \qquad \left(\nabla_{\bar{\Upsilon}}\bar{f}\left(x, U(0)\right)\right)_{33} >0,
\end{equation}
which may cause $\bar{\Upsilon}_{1}, \bar{\Upsilon}_{3}$ to grow exponentially. In order to curb this growth, we suppose that the inlet state satisfy
\begin{equation}\label{smallCondition}
	\frac{a^{\prime}(x)}{a(x)}c_{0} < \mathcal{E},
\end{equation} 
for some small constant $\mathcal{E}$. Here $c_0$ is defined by
\begin{equation}\label{staConst}
	c_{0}^{2} = \left(\gamma -1\right)\tilde{B} :=  \tilde{c}_{-}(0)^{2} + \dfrac{\gamma-1}{2}\tilde{u}_{-}(0)^{2} = \tilde{c}^{2} + \dfrac{\gamma -1}{2}\tilde{u}^{2}.
\end{equation}
The last equality is due to the Bernoulli's law. Since \eqref{staConst} guarantees that $c_{0} > \max_{x} \tilde{c}(x)$, one can see that  \eqref{smallCondition} implies
\begin{equation}
	\Vert \nabla_{\bar{\Upsilon}}\bar{f}\left(x, U(0)\right) \Vert < C_{g}\mathcal{E}.
\end{equation}
The constant $C_{g}$  only depends on the inlet boundary data.
\par At the end of this subsection, we write the system in the form of wave decomposition. Multiplying \eqref{equivSys} by the left eigenvectors $l_{i}, i = 1,2,3$ from left, one can get
\begin{equation}\label{waveDe}
	\begin{aligned}
		\partial_{t}\bar{\Upsilon}_{1} + \lambda_{1}\partial_{x}\bar{\Upsilon}_{1} &=  -\frac{a^{\prime}(x)}{a(x)}f_{1,L} -\frac{a^{\prime}(x)}{a(x)}f_{1,NL} - l_{12}\left(\partial_{t}\bar{\Upsilon}_{2} + \lambda_{1}\partial_{x}\bar{\Upsilon}_{2}\right) \\
		&\qquad\qquad\qquad\qquad -l_{13}\left[\partial_{t}\bar{\Upsilon}_{3} + \lambda_{1}\partial_{x}\bar{\Upsilon}_{3} +\frac{a^{\prime}(x)}{a(x)}\left(f_{3,L} + f_{3,NL}\right)\right], \\
		\partial_{t}\bar{\Upsilon}_{2} + \lambda_{2}\partial_{x}\bar{\Upsilon}_{2} &= 0,   \\
		\partial_{t}\bar{\Upsilon}_{3} + \lambda_{3}\partial_{x}\bar{\Upsilon}_{3} &= -\frac{a^{\prime}(x)}{a(x)}f_{3,L} - \frac{a^{\prime}(x)}{a(x)}f_{3,NL} - l_{32}\left(\partial_{t}\bar{\Upsilon}_{2} + \lambda_{3}\partial_{x}\bar{\Upsilon}_{2}\right)\\
		&\qquad\qquad\qquad\qquad
		- l_{31}\left[\partial_{t}\bar{\Upsilon}_{1} + \lambda_{3}\partial_{x}\bar{\Upsilon}_{1} +\frac{a^{\prime}(x)}{a(x)}\left(f_{1,L} + f_{1,NL}\right)\right].
	\end{aligned}
\end{equation}
In an obvious way, the boundary condition \eqref{a5} with \eqref{a8}--\eqref{a13} can be rewritten as
\begin{align}
	\bar{\Upsilon}_{1}(t,0)=\bar{\Upsilon}_{1,b-}(t),\quad &\bar{\Upsilon}_{2}(t,0)=\bar{\Upsilon}_{2,b-}(t),\quad \bar{\Upsilon}_{3}(t,0)=\bar{\Upsilon}_{3,b-}(t),\label{inletBoundary}\\
	&\bar{\Upsilon}_{1}(t,L)=-\bar{\Upsilon}_{3}(t,L) + \bar{\Upsilon}_{1,b+}(t), \label{outletBoundary}
\end{align}
with some constant $C_{b}>0$ such that
\begin{align}\label{boundarySmall}
	\|\bar{\Upsilon}_{1,b-}\|_{C^{1}(\mathbb{R}_{+})}+\|\bar{\Upsilon}_{2,b-}\|_{C^{1}(\mathbb{R}_{+})}+\|\bar{\Upsilon}_{3,b-}\|_{C^{1}(\mathbb{R}_{+})}+\|\bar{\Upsilon}_{1,b+}\|_{C^{1}(\mathbb{R}_{+})}\leq C_{b}\epsilon,
\end{align}
and $\mathcal{T}$--periodicity
\begin{align}
	\bar{\Upsilon}_{i,b-}(t+\mathcal{T})=\bar{\Upsilon}_{i,b-}(t), \quad i = 1,2,3, \quad \bar{\Upsilon}_{1,b+}(t+\mathcal{T})=\bar{\Upsilon}_{1,b+}(t). \label{boundaryPeriodicity}
\end{align}
\subsection{Time--periodic solutions in the supersonic region}
We consider flows in the supersonic region, for which all the three eigenvalues are positive. Building upon the characteristic analysis framework~\cite{Ma,Rauch,Yuan}, one has the follow lemma.
\begin{lemma}\label{existSuper}
	There exists a constant $\epsilon_{0} > 0$ such that for any $\epsilon \in (0,\epsilon_{0})$ and any inlet condition of the form \eqref{a8} satisfying \eqref{a13} and \eqref{a10}, the boundary problem \eqref{a1} and \eqref{a8} admits a unique time--periodic solution $(\rho_{-}^{(\mathcal{T})},u_{-}^{(\mathcal{T})},p_{-}^{(\mathcal{T})})(t,x)$ on the extended region $(t,x) \in \mathbb{R}_{+}\times[0,\tilde{x}+\delta]$ for some $\delta>0$ satisfying
	\begin{align}
		&\rho_{-}^{(\mathcal{T})}(t+\mathcal{T},x)=\rho_{-}^{(\mathcal{T})}(t,x), ~~ u_{-}^{(\mathcal{T})}(t+\mathcal{T},x)=u_{-}^{(\mathcal{T})}(t,x),
		~~ p_{-}^{(\mathcal{T})}(t+\mathcal{T},x)=p_{-}^{(\mathcal{T})}(t,x),\label{c1}\\
		&\|\rho_{-}^{(\mathcal{T})}(t,x)-\tilde{\rho}_{-}(x)\|_{C^{1}}+\|u_{-}^{(\mathcal{T})}(t,x)-\tilde{u}_{-}(x)\|_{C^{1}}
		+\|p_{-}^{(\mathcal{T})}(t,x)-\tilde{p}_{-}(x)\|_{C^{1}}<C_{l}\epsilon.\label{c2}
	\end{align}
\end{lemma}
The proof is based on the existence of semi--global classical solutions for hyperbolic conservation laws. We omit the details here. Indeed, to avoid the loss of derivative count, one can rewrite the system in form of wave decomposition \eqref{waveDe}. Next, one can exchange the roles of $t$ and $x$, and periodically extend the boundary conditions \eqref{inletBoundary}. In this way, the problem is transformed into a Cauchy problem for a hyperbolic conservation law. Then, the lemma can be proved by the classical characteristic method and iteration argument.

\section{Structural Dissipation of the Shock Boundary}\label{sectionDissi}
In this section, we investigate the shock boundary condition and its structural dissipation to the flows in the subsonic region.
\subsection{Shock boundary conditions}\label{boundsec}
In this subsection, we derive the shock boundary conditions for the new set of variables. It is known that in the subsonic domain, the 1--wave $\Upsilon_{1}$ reaches the shock boundary from the right, while the 2--wave $\Upsilon_{2}$ and the 3--wave $\Upsilon_{3}$ are generated at the shock. This mechanism can be determined by three equations of the R--H condition
\begin{equation}\label{RH}
	\left \{
	\begin{aligned}
		&[\rho]v = [\rho u],  \\
		&[\rho u]v = \left[\rho u^{2} + p\right],   \\
		&[\rho E]v  = \left[\rho E u + pu\right],
	\end{aligned}
	\right . 
\end{equation}
Here and below we denote the left/right limits of some function $f$ at the position of the background shock $\tilde{x}$ by $f_{l} = \tilde{f}(\tilde{x}-)$, $f_{r} = \tilde{f}(\tilde{x}+)$ the jump across the shock by $[f] = f_{r} - f_{l}$ and the shock speed by $v$.
\par One can eliminate the variable $v$ and use the Lax entropy condition
\begin{equation}\label{lax}
	u_{r} < u_{l}, \qquad \rho_{l} < \rho_{r} <\dfrac{\gamma+1}{\gamma-1} {\rho}_{l}
\end{equation}
to get
\begin{equation}\label{RHcon}
	\left \{
	\begin{aligned}
		&u_{r} - u_{l} = - \sqrt{p_{r} -p_{l}}\dfrac{\sqrt{\rho_{r} - \rho_{l}}}{\sqrt{\rho_{r}\rho_{l}}},  \\
		&p_{r} - p_{l} = \dfrac{-\gamma p_{l}\left(\rho_{r} - \rho_{l}\right)}{\frac{\gamma -1}{2}\rho_{r} -\frac{\gamma +1}{2}\rho_{l}} = \dfrac{\gamma p_{r}\left(\rho_{r} - \rho_{l}\right)}{\frac{\gamma +1}{2}\rho_{r} -\frac{\gamma -1}{2}\rho_{l}}.
	\end{aligned}
	\right . 
\end{equation}
For a fixed background solution and known left states $(\rho_{l}, u_{l}, p_{l})$, the two equations \eqref{RHcon} have three unknown variables $(\rho_{r}, u_{r}, p_{r})$, equivalently the perturbations $(\bar{\rho}_{r}, \bar{u}_{r}, \bar{p}_{r})$ or $\left(\bar{\Upsilon}_{1}, \bar{\Upsilon}_{2}, \bar{\Upsilon}_{3}\right)$. We shall determine the implicit function $\left(\bar{\Upsilon}_{2}, \bar{\Upsilon}_{3}\right) = \left(\bar{\Upsilon}_{2}, \bar{\Upsilon}_{3}\right)\left(x, \bar{\Upsilon}_{1}\right)$. To this end, we rewrite \eqref{RHcon} as
\begin{equation}\label{rh1}
	\left \{
	\begin{aligned}
		&0 = H_{1}\left(\rho_{l}, u_{l}, p_{l}, \rho_{r}, u_{r}, p_{r}\right) := \sqrt{p_{r} - p_{l}}\sqrt{\rho_{r}-\rho_{l}} + \sqrt{\rho_{r}\rho_{l}}\left(u_{r} - u_{l}\right),  \\
		&0 = H_{2}\left(\rho_{l}, u_{l}, p_{l}, \rho_{r}, u_{r}, p_{r}\right) := \left(\frac{\gamma -1}{2}\rho_{r} -\frac{\gamma +1}{2}\rho_{l}\right)\left(p_{r} - p_{l} \right) + \gamma p_{l}\left(\rho_{r} - \rho_{l}\right).
	\end{aligned}
	\right . 
\end{equation}
Using the change of variables, one can write
\begin{equation}\label{rh2}
	\left \{
	\begin{aligned}
		&0 = G_1\left(x,\bar{\Upsilon}, \bar{\rho}_{l}, \bar{u}_{l}, \bar{p}_{l}\right) := H_{1}\left(\rho_{l}, u_{l}, p_{l}, \rho_{r}, u_{r}, p_{r}\right) ,  \\
		&0 = G_2\left(x,\bar{\Upsilon}, \bar{\rho}_{l}, \bar{u}_{l}, \bar{p}_{l}\right) := H_{2}\left(\rho_{l}, u_{l}, p_{l}, \rho_{r}, u_{r}, p_{r}\right).
	\end{aligned}
	\right . 
\end{equation}
To apply the implicit function theorem, the goal is to verify that the Jacobi matrix
\begin{equation}\label{jacobi1}
	\dfrac{\partial\left(G_{1},G_{2}\right)}{\partial\left(\bar{\Upsilon}_{2}, \bar{\Upsilon}_{3}\right)}\bigg|_{\bar{\Upsilon}=0} = \dfrac{\partial\left(H_{1},H_{2}\right)}{\partial\left(\rho_{r}, u_{r}, p_{r}\right)}\bigg|_{U=\tilde{U}}\cdot\dfrac{\partial\left(\rho_{r}, u_{r}, p_{r}\right)}{\partial\left(\bar{\Upsilon}_{2}, \bar{\Upsilon}_{3}\right)}\bigg|_{\bar{\Upsilon}=0}
\end{equation}
is invertible. 
\par From \eqref{rh1}, \eqref{rh2} and \eqref{RHcon}, one has
\begin{equation}\label{devi}
	\begin{split}
		\dfrac{\partial\left(H_{1},H_{2}\right)}{\partial\left(\rho_{r}, u_{r}, p_{r}\right)}\bigg|_{U=\tilde{U}} 
		&= 
		\begin{pmatrix}
			\dfrac{\sqrt{\tilde{p}_{r} - \tilde{p}_{l}}}{2\sqrt{\tilde{\rho}_{r} - \tilde{\rho}_{l}}} + \left(\tilde{u}_{r} - \tilde{u}_{l}\right)\dfrac{\sqrt{\tilde{\rho}_{l}}}{2\sqrt{\tilde{\rho}_{r}}} & \sqrt{\tilde{\rho}_{r}\tilde{\rho}_{l}} & \dfrac{\sqrt{\tilde{\rho}_{r} - \tilde{\rho}_{l}}}{2\sqrt{\tilde{p}_{r} - \tilde{p}_{l}}} \\[2ex]
			\dfrac{\gamma -1}{2}\left(\tilde{p}_{r} - \tilde{p}_{l}\right) + \gamma \tilde{p}_{l} & 0 & \dfrac{\gamma -1}{2}\tilde{\rho}_{r} -\dfrac{\gamma +1}{2}\tilde{\rho}_{l}
		\end{pmatrix}
		\\
		&=
		\begin{pmatrix}
			\dfrac{\tilde{\rho}_{l}\sqrt{\tilde{p}_{r} - \tilde{p}_{l}}}{2\tilde{\rho}_{r}\sqrt{\tilde{\rho}_{r} - \tilde{\rho}_{l}}} & \sqrt{\tilde{\rho}_{r}\tilde{\rho}_{l}} & \dfrac{\sqrt{\tilde{\rho}_{r} - \tilde{\rho}_{l}}}{2\sqrt{\tilde{p}_{r} - \tilde{p}_{l}}} \\[2ex]
			\dfrac{\tilde{p}_{r} - \tilde{p}_{l}}{\tilde{\rho}_{r} - \tilde{\rho}_{l}}\tilde{\rho}_{l} & 0 & -\gamma \tilde{p}_{l}\dfrac{\tilde{\rho}_{r} - \tilde{\rho}_{l}}{\tilde{p}_{r} - \tilde{p}_{l}}
		\end{pmatrix}
		.
	\end{split}
\end{equation}
Thus, combining \eqref{inverVariable}, \eqref{jacobi1} and \eqref{devi}, one can see
\begin{equation}
	\dfrac{\partial\left(G_{1},G_{2}\right)}{\partial\left(\bar{\Upsilon}_{2}, \bar{\Upsilon}_{3}\right)}\bigg|_{\bar{\Upsilon}=0} =
	\begin{pmatrix}
		-\dfrac{\tilde{\rho}_{l}}{2\gamma}V & \dfrac{\tilde{\rho}_{l}}{4\gamma\sqrt{\tilde{p}_{r}}}V + \dfrac{\sqrt{\tilde{\rho}_{l}}}{2\sqrt{\gamma }} + \dfrac{\sqrt{\tilde{p}_{r}}}{4V} \\[2ex]
		-\dfrac{\tilde{\rho}_{l}\tilde{\rho}_{r}}{\gamma}V^{2} & \dfrac{\tilde{\rho}_{l}\tilde{\rho}_{r}}{2\gamma\sqrt{\tilde{p}_{r}}}V^{2} - \dfrac{\gamma \tilde{p}_{l}\sqrt{\tilde{p}_{r}}}{2V^{2}}
	\end{pmatrix}
	,
\end{equation}
where
\begin{eqnarray*}
	V = \dfrac{\sqrt{\tilde{p}_{r}-\tilde{p}_{l}}}{\sqrt{\tilde{\rho}_{r}-\tilde{\rho}_{l}}}.
\end{eqnarray*}
With the help of the second equation of \eqref{RHcon}, the determinant of this matrix reads
\begin{equation}\label{det23}
	\text{det} \left(\dfrac{\partial\left(G_{1},G_{2}\right)}{\partial\left(\bar{\Upsilon}_{2}, \bar{\Upsilon}_{3}\right)}\bigg|_{\bar{\Upsilon}=0}\right) = \dfrac{\tilde{\rho}_{l}V^{2}}{4\gamma\sqrt{\gamma}}\left[\left(\frac{\gamma +1}{2}\tilde{\rho}_{l} +\frac{3-\gamma }{2}\tilde{\rho}_{r}\right)\sqrt{\frac{\gamma +1}{2}\tilde{\rho}_{r} -\frac{\gamma -1}{2}\tilde{\rho}_{l}}+2\tilde{\rho}_{r}\sqrt{\tilde{\rho}_{l}}\right].
\end{equation}
Using \eqref{lax}, one can see
\begin{equation}
	\text{det} \left(\dfrac{\partial\left(G_{1},G_{2}\right)}{\partial\left(\bar{\Upsilon}_{2}, \bar{\Upsilon}_{3}\right)}\bigg|_{\bar{\Upsilon}=0}\right) > 0.
\end{equation}
Thus by the implicit function theorem, one can write the shock boundary condition as
\begin{equation}\label{bounrdayshock}
	\bar{\Upsilon}_{2}\left(t, x\right) = \mathscr{A}_{2}\left(x, \bar{\Upsilon}_{1}, \bar{\rho}_{l}, \bar{u}_{l}, \bar{p}_{l} \right), \qquad
	\bar{\Upsilon}_{3}\left(t, x\right) = \mathscr{A}_{3}\left(x, \bar{\Upsilon}_{1}, \bar{\rho}_{l}, \bar{u}_{l}, \bar{p}_{l} \right)
\end{equation}
where $\chi(t)$ denotes the shock position, $\mathscr{A}_{2}, \mathscr{A}_{2}$ are regarded as the implicit functions determined by \eqref{rh2}.
\par Analogous to \eqref{det23}, one has
\begin{equation}
	\nonumber
	\text{det} \left(\dfrac{\partial\left(G_{1},G_{2}\right)}{\partial\left(\bar{\Upsilon}_{2}, \bar{\Upsilon}_{1}\right)}\bigg|_{\bar{\Upsilon}=0}\right) = \dfrac{\tilde{\rho}_{l}V^{2}}{4\gamma\sqrt{\gamma}}\left[\left(\frac{\gamma +1}{2}\tilde{\rho}_{l} +\frac{3 -\gamma }{2}\tilde{\rho}_{r}\right)\sqrt{\frac{\gamma +1}{2}\tilde{\rho}_{r} -\frac{\gamma -1}{2}\tilde{\rho}_{l}}-2\tilde{\rho}_{r}\sqrt{\tilde{\rho}_{l}}\right].
\end{equation}
Thus, by implicit function theorem, one can derive
\begin{equation}\label{reflecPart}
	\begin{split}
		\dfrac{\partial \mathscr{A}_{3}}{\partial\bar{\Upsilon}_{1}}\left(\tilde{x},0,0,0,0\right) &= -\text{det} \left(\dfrac{\partial\left(G_{1},G_{2}\right)}{\partial\left(\bar{\Upsilon}_{2}, \bar{\Upsilon}_{1}\right)}\bigg|_{\bar{\Upsilon}=0}\right)\text{det} \left(\dfrac{\partial\left(G_{1},G_{2}\right)}{\partial\left(\bar{\Upsilon}_{2}, \bar{\Upsilon}_{3}\right)}\bigg|_{\bar{\Upsilon}=0}\right)^{-1} \\
		&= -\dfrac{\left(\frac{\gamma +1}{2}\tilde{\rho}_{l} +\frac{3-\gamma}{2}\tilde{\rho}_{r}\right)\sqrt{\frac{\gamma +1}{2}\tilde{\rho}_{r} -\frac{\gamma -1}{2}\tilde{\rho}_{l}}-2\tilde{\rho}_{r}\sqrt{\tilde{\rho}_{l}}}{\left(\frac{\gamma +1}{2}\tilde{\rho}_{l} +\frac{3 -\gamma}{2}\tilde{\rho}_{r}\right)\sqrt{\frac{\gamma +1}{2}\tilde{\rho}_{r} -\frac{\gamma -1}{2}\tilde{\rho}_{l}}+2\tilde{\rho}_{r}\sqrt{\tilde{\rho}_{l}}}.
	\end{split}
\end{equation}
Similarly, one can also derive 
\begin{equation}\label{nonecess}
	\dfrac{\partial \mathscr{A}_{2}}{\partial\bar{\Upsilon}_{1}}\left(\tilde{x},0,0,0,0\right) = \dfrac{\left(\gamma^{2} -1\right)\left(\tilde{\rho}_{r} - \tilde{\rho}_{l}\right)^{2}}{2\sqrt{\tilde{\rho}_{l}\tilde{p}_{r}}\left(\left(\frac{\gamma +1}{2}\tilde{\rho}_{l} +\frac{3 -\gamma}{2}\tilde{\rho}_{r}\right)\sqrt{\frac{\gamma +1}{2}\tilde{\rho}_{r} -\frac{\gamma -1}{2}\tilde{\rho}_{l}}+2\tilde{\rho}_{r}\sqrt{\tilde{\rho}_{l}}\right)}.
\end{equation}

\subsection{Structural dissipation}\label{disssec}
In this subsection, we verify the shock structural dissipation (\cref{dissipationCon}). We will estimate the magnitudes of the 3--waves generated at the shock by Hadamard's formula as
\begin{equation}\label{esti1}
	\bar{\Upsilon}_{3} \sim \left|\frac{\partial\mathscr{A}_{3}}{\partial x}\right|\Vert\chi(t) - \tilde{x}\Vert + \left|\frac{\partial\mathscr{A}_{3}}{\partial\bar{\Upsilon}_{1}}\right|\Vert\bar{\Upsilon}_{1}\Vert + \text{smaller supersonic perturbations}.
\end{equation}
The two main terms are the resonance term and the reflection term. The coefficient of the reflection term is given by \eqref{reflecPart}. To calculate the resonance term, we again consider the R--H conditions for steady shocks
\begin{equation}\label{rhBack}
	\left \{
	\begin{aligned}
		&\rho_{r}u_{r} = \rho_{l}u_{l} =: m,  \\
		&\rho_{r}u_{r}^{2} + p_{r} = \rho_{l}u_{l}^{2} + p_{l},	\\
		&\frac{1}{2}\rho_{r}u_{r}^{3} + \frac{\gamma}{\gamma-1}p_{r}u_{r} = \frac{1}{2}\rho_{l}u_{l}^{3} + \frac{\gamma}{\gamma-1}p_{l}u_{l}.
	\end{aligned}
	\right . 
\end{equation}
Here and throughout the rest of this subsection, we discuss a fixed background solution. For the sake of simplicity, we omit all the value notations $|_{\bar{\Upsilon}=0}$ or $\left(\tilde{x},0,0,0,0\right)$ and the tilde notations.
\par First of all, we present some crucial substitution formulas that will be used in the subsequent derivations.
\begin{lemma}
	For a given steady transonic shock solution $\left(\rho, u, p\right)$, which satisfying R--H conditions \eqref{rhBack} at the shock, the following equations hold.
	\begin{align}
		\rho_{r}u_{r} = \rho_{l}u_{l} =: m, \label{rhFormu1}\\
		V := \frac{\sqrt{p_{r}-p_{l}}}{\sqrt{\rho_{r}-\rho_{l}}} = \frac{m}{\sqrt{\rho_{r}\rho_{l}}}, \label{rhFormu2}\\
		p_{r}-p_{l} = -m\left(u_{r}-u_{l}\right) = \frac{\rho_{r}-\rho_{l}}{\rho_{r}\rho_{l}}m^{2}, \label{rhFormu3}\\
		mu_{l}-\gamma p_{l} = -\left(mu_{r}-\gamma p_{r}\right), \label{rhFormu4}\\
		m\left(u_{r}-u_{l}\right) = \frac{2}{\gamma+1}\left(mu_{r} - \gamma p_{r}\right). \label{rhFormu5}
	\end{align}
\end{lemma}
\begin{proof}
	\eqref{rhFormu1} is the first momentum equation of R--H conditions \eqref{rhBack}, which indicates that the momentum density across the shock wave remains unchanged. 
	\par By \eqref{RHcon} and the second equation of \eqref{rhBack}, one can see
	\begin{eqnarray*}
		p_{r} - p_{l} = -m\left(u_{r} - u_{l}\right) = m\sqrt{p_{r} -p_{l}}\dfrac{\sqrt{\rho_{r} - \rho_{l}}}{\sqrt{\rho_{r}\rho_{l}}}, 
	\end{eqnarray*}
	which yields \eqref{rhFormu2} and \eqref{rhFormu3}.
	\par Moreover, one can deduce from \eqref{rhBack} that
	\begin{equation}\label{rhSolve}
			\rho_{l} = \frac{\left(\gamma+1\right)m^{2}}{2\gamma p_{r}+\left(\gamma-1\right)mu_{r}}, \quad u_{l} = \frac{1}{\gamma+1}\left((\gamma-1)u_{r}+2\gamma\frac{p_{r}}{m}\right),	\quad p_{l} =\frac{2mu_{r}-\left(\gamma-1\right)p_{r}}{\gamma+1}.
	\end{equation}
	One can directly get \eqref{rhFormu4} and \eqref{rhFormu5} from \eqref{rhSolve}.
\end{proof}
Now we are ready to calculate
\begin{equation}
	\dfrac{\partial \mathscr{A}_{3}}{\partial x} = -\text{det} \left(\dfrac{\partial\left(G_{1},G_{2}\right)}{\partial\left(\bar{\Upsilon}_{2}, x\right)}\right)\text{det} \left(\dfrac{\partial\left(G_{1},G_{2}\right)}{\partial\left(\bar{\Upsilon}_{2}, \bar{\Upsilon}_{3}\right)}\right)^{-1}.
\end{equation}
Analogous to \eqref{devi}, one can deduce
\begin{equation}\label{devi2}
	\dfrac{\partial\left(H_{1},H_{2}\right)}{\partial\left(\rho_{l}, u_{l}, p_{l}\right)} = \begin{pmatrix}
		- \dfrac{\rho_{r}\sqrt{p_{r} - p_{l}}}{2\rho_{l}\sqrt{\rho_{r} - \rho_{l}}} & - \sqrt{\rho_{r}\rho_{l}} & - \dfrac{\sqrt{\rho_{r} - \rho_{l}}}{2\sqrt{p_{r} - p_{l}}} \\[2ex]
		-\dfrac{p_{r} - p_{l}}{\rho_{r} - \rho_{l}}\rho_{r} & 0 & \gamma p_{r}\dfrac{\rho_{r} - \rho_{l}}{p_{r} - p_{l}}
	\end{pmatrix}
	.
\end{equation}
Therefore combining \eqref{devi}, \eqref{devi2} and \eqref{b2}, one has
\begin{multline}
	\frac{\partial G_{1}}{\partial x} \left(-\frac{a^{\prime}(x)}{a(x)}\right)^{-1} = \frac{V}{2}\left(\frac{\rho_{l}u_{r}^{2}}{u_{r}^{2}-c_{r}^{2}}-\frac{\rho_{r}u_{l}^{2}}{u_{l}^{2}-c_{l}^{2}}\right) - \sqrt{\rho_{r}\rho_{l}}\left(\frac{c_{r}^{2}u_{r}}{u_{r}^{2}-c_{r}^{2}} - \frac{c_{l}^{2}u_{l}}{u_{l}^{2}-c_{l}^{2}}\right) \\
	+ \frac{1}{2V}\left(\frac{c_{r}^{2}\rho_{r}u_{r}^{2}}{u_{r}^{2}-c_{r}^{2}} - \frac{c_{l}^{2}\rho_{l}u_{l}^{2}}{u_{l}^{2}-c_{l}^{2}}\right).
\end{multline}
Using \eqref{rhFormu2} and \eqref{rhFormu3}, one can see  
\begin{equation}
	\begin{split}
		\frac{\partial G_{1}}{\partial x} \left(-\frac{a^{\prime}(x)}{a(x)}\right)^{-1} &= \frac{1}{2V}\left(\frac{\rho_{r}u_{r}^{2}u_{r}^{2}}{u_{r}^{2}-c_{r}^{2}}-\frac{\rho_{l}u_{l}^{2}u_{l}^{2}}{u_{l}^{2}-c_{l}^{2}}\right) - \frac{1}{2V}\left(\frac{\rho_{r}u_{r}^{2}c_{r}^{2}}{u_{r}^{2}-c_{r}^{2}}-\frac{\rho_{l}u_{l}^{2}c_{l}^{2}}{u_{l}^{2}-c_{l}^{2}}\right) \\
		&= \frac{1}{2V}\left(\rho_{r}u_{r}^{2} - \rho_{l}u_{l}^{2}\right) = -\frac{p_{r}-p_{l}}{2V}.
	\end{split}
\end{equation}
\par On the other hand, noticing \eqref{rhFormu2} and \eqref{RHcon}, one has
\begin{equation}
	\begin{split}
		\frac{\partial G_{2}}{\partial x} \left(-\frac{a^{\prime}(x)}{a(x)}\right)^{-1} &= V^{2}\rho_{r}\rho_{l}\left(\frac{u_{r}^{2}}{u_{r}^{2}-c_{r}^{2}}-\frac{u_{l}^{2}}{u_{l}^{2}-c_{l}^{2}}\right)- \frac{\gamma}{V^{2}}\left(p_{l}\frac{c_{r}^{2}\rho_{r}u_{r}^{2}}{u_{r}^{2}-c_{r}^{2}} - p_{r}\frac{c_{l}^{2}\rho_{l}u_{l}^{2}}{u_{l}^{2}-c_{l}^{2}}\right) \\
		&= m^{2}\left(\frac{u_{r}^{2}-c_{r}^{2}}{u_{r}^{2}-c_{r}^{2}}-\frac{u_{l}^{2}-c_{l}^{2}}{u_{l}^{2}-c_{l}^{2}}\right)+\frac{\gamma+1}{2}\left(\rho_{r}-\rho_{l}\right)\left(\frac{c_{r}^{2}\rho_{r}u_{r}^{2}}{u_{r}^{2}-c_{r}^{2}} + \frac{c_{l}^{2}\rho_{l}u_{l}^{2}}{u_{l}^{2}-c_{l}^{2}}\right) \\
		&=-\frac{\gamma-1}{2}m\left(\rho_{r}-\rho_{l}\right)\left(u_{r}+u_{l}\right).
	\end{split}
\end{equation}
Here \eqref{rhFormu4}, \eqref{rhBack} and \eqref{rhFormu5} are used for the last two equals. So far, one has obtained
\begin{equation}
	\dfrac{\partial\left(G_{1},G_{2}\right)}{\partial\left(\bar{\Upsilon}_{2}, x\right)} =
	\begin{pmatrix}
		-\dfrac{\rho_{l}}{2\gamma}V & \dfrac{a^{\prime}(x)}{a(x)}\dfrac{p_{r}-p_{l}}{2V} \\
		-\dfrac{\rho_{r}\rho_{l}}{\gamma}V^{2} & \dfrac{a^{\prime}(x)}{a(x)}\dfrac{\gamma-1}{2}m\left(\rho_{r}-\rho_{l}\right)\left(u_{r}+u_{l}\right)
	\end{pmatrix}
\end{equation}
Therefore, one has 
\begin{equation}
	\begin{split}
		\text{det} \left(\dfrac{\partial\left(G_{1},G_{2}\right)}{\partial\left(\bar{\Upsilon}_{2}, x\right)}\right)\left(-\dfrac{a^{\prime}(x)}{a(x)}\right)^{-1} &= \dfrac{\gamma -1}{4\gamma}\rho_{l}Vm\left(\rho_{r}-\rho_{l}\right)\left(u_{r}+u_{l}\right)-\dfrac{1}{2\gamma}\rho_{r}\rho_{l}V\left(p_{r}-p_{l}\right)\\
		&= \dfrac{1}{2\gamma}Vm\left(\rho_{r}-\rho_{l}\right)\left(\dfrac{\gamma-1}{2}\rho_{l}u_{r}-\dfrac{3-\gamma}{2}m\right).
	\end{split}
\end{equation}
Thus,
\begin{equation}\label{pA3px}
	\dfrac{\partial \mathscr{A}_{3}}{\partial x} = \dfrac{a^{\prime}(x)}{a(x)}\frac{\sqrt{\gamma}\sqrt{\rho_{r}\rho_{l}}\left(\rho_{r}-\rho_{l}\right)\left(\left(\gamma-1\right)u_{r}-\left(3-\gamma\right)u_{l}\right)}{\left(\frac{\gamma +1}{2}\rho_{l} +\frac{3 -\gamma}{2}\rho_{r}\right)\sqrt{\frac{\gamma +1}{2}\rho_{r} -\frac{\gamma -1}{2}\rho_{l}}+2\rho_{r}\sqrt{\rho_{l}}}.
\end{equation}
\par Next, we focus on the offset of the shock position $|\Vert\chi(t) - \tilde{x}\Vert$. Again, from the R--H conditions, one can derive the shock position satisfies some ODE
\begin{equation}\label{shockODE}
	\chi^{\prime}(t) = \mathcal{F}\left(x, \bar{\Upsilon}, \bar{\rho}_{l}, \bar{u}_{l}, \bar{p}_{l}\right).
\end{equation}
In particular, one can choose
\begin{equation}\label{speedEx}
	\chi^{\prime}(t) = u_{l}-\sqrt{\dfrac{\rho_{r}}{\rho_{l}}}\sqrt{\dfrac{p_{r}-p_{l}}{\rho_{r}-\rho_{l}}} \equiv: F\left(U_{r},U_{l}\right).
\end{equation}
Applying the change of variables, one can define
\begin{equation}
	\mathcal{F}\left(x, \bar{\Upsilon}, \bar{\rho}_{l}, \bar{u}_{l}, \bar{p}_{l}\right) = F\left(U_{r},U_{l}\right)
\end{equation}
such that \eqref{shockODE} holds. 
\begin{remark}
	In fact, we have infinitely many expressions of $\mathcal{F}$ from the R--H conditions.
\end{remark}
We will estimate the amplitude of the shock position by \eqref{bb19}:
\begin{equation}\label{estiODE}
	\Vert\chi(t) - \tilde{x}\Vert \sim \dfrac{1}{\left|\frac{\partial \mathcal{F}}{\partial x}\right|}\sum_{i=1}^{3}\left|\frac{\partial \mathcal{F}}{\partial \bar{\Upsilon}_{i}}\right|\Vert\bar{\Upsilon}_{i}\Vert+ \text{supersonic perturbation terms}.
\end{equation}
Using \eqref{speedEx}, one has
\begin{equation}
	\nonumber
	\frac{\partial \mathcal{F}}{\partial x} = u_{l}^{\prime} - V\left(\dfrac{\rho_{r}^{\prime}}{2\sqrt{\rho_{r}\rho_{l}}}- \dfrac{\rho_{l}^{\prime}}{2\rho_{l}\sqrt{\rho_{l}}}\right) -\sqrt{\dfrac{\rho_{r}}{\rho_{l}}}\left[\dfrac{p_{r}^{\prime}-p_{l}^{\prime}}{2\sqrt{p_{r}-p_{l}}\sqrt{\rho_{r}-\rho_{l}}}-\dfrac{\left(\rho_{r}^{\prime}-\rho_{l}^{\prime}\right)\sqrt{p_{r}-p_{l}}}{2\sqrt{\rho_{r}-\rho_{l}}\left(\rho_{r}-\rho_{l}\right)}\right]
\end{equation}
Applying \eqref{b2}, one can get
\begin{equation}
	\nonumber
	\frac{\partial \mathcal{F}}{\partial x}\left(-\dfrac{a^{\prime}(x)}{a(x)}\right)^{-1} = \dfrac{-mc_{l}^{2}}{mu_{l}-\gamma p_{l}} + \dfrac{u_{r}u_{l}^{2}\left(\rho_{r}+\rho_{l}\right)}{2\left(mu_{l}-\gamma p_{l}\right)} - \dfrac{2u_{l}m^{2}-\rho_{r}m\left(c_{r}^{2}+c_{l}^{2}\right)}{2\left(\rho_{r}-\rho_{l}\right)\left(mu_{l}-\gamma p_{l}\right)}.
\end{equation}
Then \eqref{rhFormu1}--\eqref{rhFormu4} yield
\begin{equation}\label{pFpx}
	\frac{\partial \mathcal{F}}{\partial x} = -\dfrac{a^{\prime}(x)}{a(x)}\dfrac{u_{l}}{2}<0.
\end{equation}
This, combining the obvious fact $\mathcal{F}(\tilde{x},\mathbf{0}) = 0$, enables one to apply \Cref{lemmaODE1} and \Cref{lemmaODE2} to the ODE \eqref{shockODE}. Moreover, with \eqref{smallCondition} and \eqref{staConst}, one can see
\begin{equation}\label{Fxsmall}
	\underline{C}\mathcal{E} \le \left|\frac{\partial \mathcal{F}}{\partial x}\right| \le \bar{C}\mathcal{E},
\end{equation}
for some constants $\underline{C}$ and $\bar{C}$.
\par On the other hand, one has
\begin{eqnarray*}
	\frac{\partial \mathcal{F}}{\partial \bar{\Upsilon}_{1}} = \frac{\partial \mathcal{F}}{\partial \rho_{r}}\cdot\frac{\partial \rho_{r}}{\partial \bar{\Upsilon}_{1}} + \frac{\partial \mathcal{F}}{\partial p_{r}}\cdot\frac{\partial p_{r}}{\partial \bar{\Upsilon}_{1}}.
\end{eqnarray*}
Direct calculation yields
\begin{equation}
	\frac{\partial \mathcal{F}}{\partial \rho_{r}} = \dfrac{V}{2\sqrt{\rho_{r}\rho_{l}}}\cdot\dfrac{\rho_{l}}{\rho_{r}-\rho_{l}}, \qquad
	\frac{\partial \mathcal{F}}{\partial p_{r}} = -\sqrt{\dfrac{\rho_{r}}{\rho_{l}}}\cdot\dfrac{1}{2V\left(\rho_{r}-\rho_{l}\right)}.
\end{equation}
Combining \eqref{inverVariable}, one can get
\begin{equation}\label{pFpU1}
	\frac{\partial \mathcal{F}}{\partial \bar{\Upsilon}_{1}} = - \dfrac{\left(\gamma+1\right)u_{l}}{8\gamma \sqrt{p_{r}}}.
\end{equation}
Similarly, 
\begin{equation}\label{pFpU3}
	\frac{\partial \mathcal{F}}{\partial \bar{\Upsilon}_{2}}= -\dfrac{m}{2\gamma\left(\rho_{r}-\rho_{l}\right)},\quad\frac{\partial \mathcal{F}}{\partial \bar{\Upsilon}_{3}} = - \dfrac{\left(\gamma+1\right)u_{l}}{8\gamma \sqrt{p_{r}}}.
\end{equation}
\par In the view of \eqref{esti1} and \eqref{estiODE}, our goal is that the coefficient carries dissipation, i.e.
\begin{equation}\label{dissipationCon}
	\tilde{\theta}_{d} := \left|\frac{\partial\mathscr{A}_{3}}{\partial x}\right|\cdot\left|\frac{\partial \mathcal{F}}{\partial x}\right|^{-1}\cdot\left(\left|\frac{\partial \mathcal{F}}{\partial \bar{\Upsilon}_{1}}\right|+ \left|\frac{\partial \mathcal{F}}{\partial \bar{\Upsilon}_{3}}\right|\right) +  \left|\frac{\partial\mathscr{A}_{3}}{\partial\bar{\Upsilon}_{1}}\right| < 1.
\end{equation}
One can observe that the dissipation \eqref{dissipationCon} is not necessarily valid. Since one can let $\gamma \to 1^{+}$, \eqref{dissipationCon} yields
	\begin{equation}
		\sqrt{2} - 1 < M_{r} := \dfrac{u_{r}}{c_{r}} < 1.
	\end{equation}
	This is consistent with the restriction in \cite{ZHANG}. However, \eqref{dissipationCon} is indeed valid when $\gamma$ is larger. We have 
\begin{lemma}\label{lemmaDissipa}
	There exits a critical $\gamma_{c} \in \left(1, \frac{5}{3}\right)$, such that:
	\begin{enumerate}
		\item[a.] If $\gamma \in \left(\gamma_{c},3\right)$, the shock boundary is dissipative for any steady state, i.e., \eqref{dissipationCon} holds.
		\item[b.] If $\gamma \in \left(1, \gamma_{c}\right]$, the shock boundary is dissipative if the background transonic shock solution satisfying
		\begin{equation}\label{conditionM}
			\tilde{M}_{c} < M_{r} < 1,
		\end{equation}
		for some constant $\tilde{M}_{c} = \tilde{M}_{c}(\gamma)$.
	\end{enumerate}
\end{lemma}
\begin{proof}
	\begin{description}[leftmargin=0cm]
		\item[\textbf{Case 1}] First, we discuss the case when 
		\begin{equation}\label{caseOne}
			\left(\frac{\gamma +1}{2}\rho_{l} +\frac{3 -\gamma}{2}\rho_{r}\right)\sqrt{\frac{\gamma +1}{2}\rho_{r} -\frac{\gamma -1}{2}\rho_{l}}-2\rho_{r}\sqrt{\rho_{l}} >0.
		\end{equation}
		Denoting the compression ratio by $R=\frac{\rho_{r}}{\rho_{l}}\in \left(1, \frac{\gamma+1}{\gamma-1}\right)$, one can see \eqref{caseOne} is always valid if $\gamma \in \left(1, \frac{5}{3}\right)$. In fact, \eqref{caseOne} is not valid if and only if 
		\begin{eqnarray*}
			\min \{1, \dfrac{\gamma^{2} -1}{\left(3-\gamma\right)^{2}}\} < R < \max \{1, \dfrac{\gamma^{2} -1}{\left(3-\gamma\right)^{2}}\}.
		\end{eqnarray*}
		On the other hand, $1< \dfrac{\gamma^{2} -1}{\left(3-\gamma\right)^{2}}$ if and only if $\gamma > \frac{5}{3}$. 
		\par Substituting \eqref{reflecPart}, \eqref{pA3px}, \eqref{pFpU1}, \eqref{pFpU3}, \eqref{pFpx}, the inequality \eqref{dissipationCon} is equivalent to
		\begin{equation}\label{dissipationCon1}
			\dfrac{\gamma+1}{2}\left(u_{l} - u_{r}\right)\left[\left(3 - \gamma\right)\rho_{r} - \left(\gamma -1\right)\rho_{l}\right] < 4\rho_{r}c_{r}.
		\end{equation}
		One can directly get from \eqref{rhSolve} that
		\begin{equation}\label{postMechpre}
			M_{r}^{2} = \dfrac{\left(\gamma -1\right)M_{l}^{2}+2}{2\gamma M_{l}^{2}-\left(\gamma -1\right)},
		\end{equation}
		where $M_{l,r}$ are the Mach numbers. Thus for supersonic inflow $M_{l} > 1$, the post--shock Mach number is monotonically related to the pre--shock Mach number and
		\begin{equation}
			\dfrac{\gamma-1}{2\gamma} < M_{r}^{2} <1.
		\end{equation}
		Using \eqref{postMechpre}, \eqref{dissipationCon1} can be reduced to
		\begin{equation}
			\dfrac{3-\gamma}{M_{r}^{2}}+\left(\gamma -1\right)\dfrac{2\gamma M_{r}^{2}-\gamma +1}{\left(\gamma -1\right)M_{r}^{2}+2} - 2 < \dfrac{4}{M_{r}},
		\end{equation}
		equivalently
		\begin{equation}
			\begin{split}
				&P(x) := \left(\gamma-1\right)^{2}x^{4} - 2\left(\gamma -1\right)x^{3} - \left(\gamma^{2} -3\gamma +4\right)x^{2} - 4x +3-\gamma,\\
				&P(M_{r})<0.
			\end{split}
		\end{equation}
		One can see that $P(x)$ is monotonically decreasing as $\gamma\in \left(1,\frac{5}{3}\right)$. Therefore, $P(M_{r})<0$ if and only if $P(\sqrt{\frac{\gamma -1}{2\gamma}})<0$. Let $\gamma_{c} \in \left(0,\frac{5}{3}\right)$ satisfy $P(\sqrt{\frac{\gamma_{c} -1}{2\gamma_{c}}}) = 0$ and $M_{c}$ be the zero point of the equation $P(M_{c}) =0$. One can see the part b. of the lemma.
		\item[\textbf{Case 2}] Second, we discuss the case when
		\begin{equation}\label{caseTwo}
			\left(\frac{\gamma +1}{2}\rho_{l} +\frac{3 -\gamma}{2}\rho_{r}\right)\sqrt{\frac{\gamma +1}{2}\rho_{r} -\frac{\gamma -1}{2}\rho_{l}}-2\rho_{r}\sqrt{\rho_{l}} <0.
		\end{equation}
		Substituting \eqref{reflecPart}, \eqref{pA3px}, \eqref{pFpU1}, \eqref{pFpU3}, \eqref{pFpx}, the inequality \eqref{dissipationCon} is equivalent to
		\begin{equation}\label{dissipationCon2}
			\dfrac{\gamma+1}{2c_{r}}\sqrt{\rho_{l}}\left(u_{l} - u_{r}\right)\left|\left(\gamma -1\right)\rho_{l} -\left(3 - \gamma\right)\rho_{r}\right| < 2\left(\frac{\gamma +1}{2}\rho_{l} +\frac{3 -\gamma}{2}\rho_{r}\right)\sqrt{\frac{\gamma +1}{2}\rho_{r} -\frac{\gamma -1}{2}\rho_{l}},
		\end{equation}
		which can be rewritten as
		\begin{equation}\label{disCon2}
			\left(\gamma+1\right)\left(R-1\right)\left|\left(\gamma-1\right)- \left(3-\gamma\right)R\right| < \left[\left(\gamma+1\right)R - \left(\gamma-1\right)\right]\left[\gamma+1+\left(3-\gamma\right)R\right].
		\end{equation}
		One can easily see that \eqref{disCon2} is always valid since $R >1$.
	\end{description} 
\end{proof}
\begin{remark}
	Under the method of this paper, one can get
	\begin{equation}
		\gamma_{c} \approx 1.347.
	\end{equation}
	In particular, for ordinary air $(\gamma=1.4)$, any transonic shocks indeed have structural dissipation.
\end{remark}
\par In fact, by continuity and boundedness of the background solution, we have a stronger conclusion that there exist a constant $\theta_{d}$ depend only on $m_{a}$, $M_{a}$, $L$ and inlet data \eqref{a7} such that
\begin{eqnarray*}
	\tilde{\theta}_{d} < \theta_{d} <1.
\end{eqnarray*}
\par Furthermore, according to the ODE for Mach number, one can get
\begin{equation}
	\mathscr{M}(\tilde{M}(x)) = \frac{a(x)}{a(0)}\mathscr{M}(\tilde{M}(0)),
\end{equation} 
where 
\begin{eqnarray*}
	\mathscr{M}(M) := \frac{1}{M}\left(\frac{2}{\gamma+1}\left(1+\frac{\gamma-1}{2}M^{2}\right)\right)^{\frac{\gamma+1}{2(\gamma-1)}}.
\end{eqnarray*}
Thus, one can see that the condition \eqref{conditionM} can be guaranteed by \eqref{AA11} for some constant $M_{crit}$ depending on $\gamma$, $m_{a}$, $M_{a}$ and the inlet data \eqref{a7}.
\par For the completeness of the argument, we further explain the dissipative structure of the entropy wave. Similarly, one has
\begin{equation}
	\bar{\Upsilon}_{2} \sim \left|\frac{\partial\mathscr{A}_{2}}{\partial x}\right|\Vert\chi(t) - \tilde{x}\Vert + \left|\frac{\partial\mathscr{A}_{2}}{\partial\bar{\Upsilon}_{1}}\right|\Vert\bar{\Upsilon}_{1}\Vert + \text{smaller supersonic perturbations}.
\end{equation}
Since the entropy wave has a different physical dimension from acoustic waves, after appropriate scaling
\begin{align}
	&x' = \nu x, \quad t' = \nu^{a}t, \nonumber\\
	&\rho'(t',x') = \nu^{b}\rho(t,x), \quad u'(t',x') = \nu^{1-a}u(t,x), \quad p'(t',x') = \nu^{2-2a+b}p(t,x), \label{sc}
\end{align}
the system \eqref{a1} is invariant for any $a,b,\nu >0$. Thus, without loss of generality, we can assume that the entropy wave has a dissipative structure at the shock in our analysis.
\subsection{Reformulation of the system}
In order to make the boundary conditions have formal dissipation, we apply the following scaling transformation
\begin{equation}\label{scaling}
	\breve{\mathbf{\Upsilon}} = \left(\breve{\Upsilon}_{1}, \breve{\Upsilon}_{2}, \breve{\Upsilon}_{3}\right)^{\top} := \left(\alpha_{1}\bar{\Upsilon}_{1}, \alpha_{2}\bar{\Upsilon}_{2}, \alpha_{3}\bar{\Upsilon}_{3}\right)^{\top},
\end{equation}
where we set that $\alpha_{1}=1$, and that $\alpha_{2}$, $\alpha_{3}$ are some constants satisfying
\begin{equation}\label{beta}
	\alpha_{3}>1, \qquad \alpha_{3}\theta_{d}+\frac{\alpha_{3}}{\alpha_{2}}<1.
\end{equation}
\par Now, we conclude that in the subsonic domain, the problem takes the following form
\begin{align}
	\partial_{t}\breve{\Upsilon}_{1}+\lambda_{1}\partial_{x}\breve{\Upsilon}_{1} &= g_{1}(x, \breve{\mathbf{\Upsilon}}) - \frac{l_{12}}{\alpha_{2}}\left(\partial_{t}\breve{\Upsilon}_{2}+\lambda_{1}\partial_{x}\breve{\Upsilon}_{2}\right) - \frac{l_{13}}{\alpha_{3}}\left(\partial_{t}\breve{\Upsilon}_{3}+\lambda_{1}\partial_{x}\breve{\Upsilon}_{3}\right), \\
	\partial_{t}\breve{\Upsilon}_{2}+\lambda_{2}\partial_{x}\breve{\Upsilon}_{2} &= 0, \\
	\partial_{t}\breve{\Upsilon}_{3}+\lambda_{3}\partial_{x}\breve{\Upsilon}_{3} &= \alpha_{3}g_{3}(x, \breve{\mathbf{\Upsilon}}) - \frac{\alpha_{3} l_{32}}{\alpha_{2}}\left(\partial_{t}\breve{\Upsilon}_{2}+\lambda_{3}\partial_{x}\breve{\Upsilon}_{2}\right) - \alpha_{3} l_{31}\left(\partial_{t}\breve{\Upsilon}_{1}+\lambda_{3}\partial_{x}\breve{\Upsilon}_{1}\right), \label{system3}
\end{align}
where
\begin{align}
	\nonumber
	g_{1}(x, \breve{\mathbf{\Upsilon}}) = -\frac{a^{\prime}(x)}{a(x)}\left(f_{1,L}+f_{1,NL}+l_{13}f_{3,L}+l_{13}f_{3,NL}\right), \\
	\nonumber
	g_{3}(x, \breve{\mathbf{\Upsilon}}) = -\frac{a^{\prime}(x)}{a(x)}\left(f_{3,L}+f_{3,NL}+l_{31}f_{1,L}+l_{31}f_{1,NL}\right),
\end{align}
with shock boundary and outlet boundary conditions
\begin{align}
	\breve{\Upsilon}_{2}\left(t, \chi(t)\right) &=  \alpha_{2}\mathscr{A}_{2}\left(\chi(t), \breve{\Upsilon}_{1}(t,\chi(t)), \bar{\rho}_{l}(t,\chi(t)), \bar{u}_{l}(t,\chi(t)), \bar{p}_{l}(t,\chi(t)) \right), \label{shockBoundary2}\\
	\breve{\Upsilon}_{3}\left(t, \chi(t)\right) &=  \alpha_{3}\mathscr{A}_{3}\left(\chi(t), \breve{\Upsilon}_{1}(t,\chi(t)), \bar{\rho}_{l}(t,\chi(t)), \bar{u}_{l}(t,\chi(t)), \bar{p}_{l}(t,\chi(t)) \right), \label{shockBoundary3}\\
	\breve{\Upsilon}_{1}\left(t, L\right) &= -\frac{1}{\alpha_{3}}\breve{\Upsilon}_{3}\left(t, L\right) + \bar{\Upsilon}_{1,b+}(t). \label{outletBoundary1}
\end{align}
In addition, the shock position satisfies the ODE
\begin{equation}\label{shockposition}
	\chi^{\prime}(t) = \mathcal{F}\left(\chi(t), \left(\breve{\Upsilon}_{1}, \frac{1}{\alpha_{2}}\breve{\Upsilon}_{2}, \frac{1}{\alpha_{3}}\breve{\Upsilon}_{3}\right)(t,\chi(t)), \bar{\rho}_{l}(t,\chi(t)), \bar{u}_{l}(t,\chi(t)), \bar{p}_{l}(t,\chi(t))\right).
\end{equation}
\section{Existence of Time--periodic Transonic Shock Solutions}\label{s3}
In this section, we prove our main \Cref{t1}.
\subsection{A fractional--step iteration scheme}
We start with building a fractional--step iteration scheme to construct a sequence of approximate solutions in the subsonic region. To set up, let
\begin{align}\label{start}
\breve{\mathbf{\Upsilon}}^{(0)}(t,x)=\mathbf{0}=(0,0,0)^{\top},\quad\chi^{(0)}(t)\equiv\tilde{x}.
\end{align}
For any $k\in\mathbb{N}_{+}$, suppose $\left(\breve{\mathbf{\Upsilon}}^{(k-1)}, \chi^{(k-1)}(t)\right)$ is given, we first construct $\breve{\mathbf{\Upsilon}}^{(k)}$ as the solution to the linearized system
\begin{align}
	\partial_{t}\breve{\Upsilon}_{1}^{(k)}+\lambda_{1}^{(k-1)}\partial_{x}\breve{\Upsilon}_{1}^{(k)} &= g_{1}(x, \breve{\mathbf{\Upsilon}}^{(k-1)}) - \sum_{i=2,3}\frac{l_{1i}^{(k-1)}}{\alpha_{i}}\left(\partial_{t}\breve{\Upsilon}_{i}^{(k-1)}+\lambda_{1}^{(k-1)}\partial_{x}\breve{\Upsilon}_{i}^{(k-1)}\right),\label{liearSys1}\\
	\partial_{t}\breve{\Upsilon}_{2}^{(k)}+\lambda_{2}^{(k-1)}\partial_{x}\breve{\Upsilon}_{2}^{(k)} &= 0, \label{liearSys2}\\
	\partial_{t}\breve{\Upsilon}_{3}^{(k)}+\lambda_{3}^{(k-1)}\partial_{x}\breve{\Upsilon}_{3}^{(k)} &= \alpha_{3}g_{3}(x, \breve{\mathbf{\Upsilon}}^{(k-1)}) - \sum_{i=1,2}\frac{\alpha_{3}l_{3i}^{(k-1)}}{\alpha_{i}}\left(\partial_{t}\breve{\Upsilon}_{i}^{(k-1)}+\lambda_{3}^{(k-1)}\partial_{x}\breve{\Upsilon}_{i}^{(k-1)}\right), \label{liearSys3}
\end{align}
with the linearized boundary condition
\begin{align}
&\breve{\Upsilon}_{2}^{(k)}(t,\chi^{(k-1)}(t))=\mathscr{A}_{2}^{(k-1)}\left[t,\chi^{(k-1)}(t)\right],\quad
\breve{\Upsilon}_{3}^{(k)}(t,\chi^{(k-1)}(t))=\mathscr{A}_{3}^{(k-1)}\left[t,\chi^{(k-1)}(t)\right],\label{c9}\\
&\breve{\Upsilon}_{1}^{(k)}(t,L)= -\frac{1}{\alpha_{3}}\breve{\Upsilon}_{3}^{(k-1)}\left(t, L\right) + \bar{\Upsilon}_{1,b+}(t).\label{c7}
\end{align}
Here we use the notation
\begin{eqnarray*}
	\lambda_{i}^{(k-1)} := \lambda_{i}(x, \breve{\mathbf{\Upsilon}}^{(k-1)}), \quad l_{ij}^{(k-1)} := l_{ij}(x, \breve{\mathbf{\Upsilon}}^{(k-1)}), \quad i,j =1,2,3,
\end{eqnarray*}
and
\begin{multline}
	\nonumber
	\mathscr{A}_{2}^{(k-1)}\left[t,\chi^{(k-1)}(t)\right] \\
	:= \alpha_{2}\mathscr{A}_{2}\left(\chi^{(k-1)}(t), \breve{\Upsilon}_{1}^{(k-1)}(t,\chi^{(k-1)}(t)), \bar{\rho}_{l}(t,\chi^{(k-1)}(t)), \bar{u}_{l}(t,\chi^{(k-1)}(t)), \bar{p}_{l}(t,\chi^{(k-1)}(t)) \right),
\end{multline}
\begin{multline}
	\nonumber
	\mathscr{A}_{3}^{(k-1)}\left[t,\chi^{(k-1)}(t)\right] \\
	:= \alpha_{3}\mathscr{A}_{3}\left(\chi^{(k-1)}(t), \breve{\Upsilon}_{1}^{(k-1)}(t,\chi^{(k-1)}(t)), \bar{\rho}_{l}(t,\chi^{(k-1)}(t)), \bar{u}_{l}(t,\chi^{(k-1)}(t)), \bar{p}_{l}(t,\chi^{(k-1)}(t)) \right).
\end{multline}
For sake of simplification, we also use notations
\begin{eqnarray*}
	\frac{\partial \mathscr{A}^{(k-1)}}{\partial x}\left[t,\chi^{(k-1)}(t)\right], \frac{\partial \mathscr{A}^{(k-1)}}{\partial \bar{\Upsilon}_{1}}\left[t,\chi^{(k-1)}(t)\right], \\
	\frac{\partial \mathscr{A}^{(k-1)}}{\partial \bar{\rho}}\left[t,\chi^{(k-1)}(t)\right], \frac{\partial \mathscr{A}^{(k-1)}}{\partial \bar{u}}\left[t,\chi^{(k-1)}(t)\right], \frac{\partial \mathscr{A}^{(k-1)}}{\partial \bar{p}}\left[t,\chi^{(k-1)}(t)\right]
\end{eqnarray*}
to denote the corresponding derivatives on the shock boundary.
\par The shock boundary \eqref{c9} is associated with the perturbation on the left of the shock $\left(\bar{\rho}_{l}, \bar{u}_{l}, \bar{p}_{l}\right)$, given by \Cref{existSuper}.
Denoting
$$\bar{\rho}_{l}(t,x) = \rho_{-}^{(\mathcal{T})}(t,x)-\tilde{\rho}_-(x),~ \bar{u}_{l}(t,x) = u_{-}^{(\mathcal{T})}(t,x)-\tilde{u}_-(x),~ \bar{p}_{l}(t,x) = p_{-}^{(\mathcal{T})}(t,x)-\tilde{p}_-(x),$$
one has
\begin{align}
\|\bar{\rho}_{l}\|_{C^{1}}+\|\bar{u}_{l}\|_{C^{1}}+\|\bar{p}_{l}\|_{C^{1}}
<C_{l}\epsilon. \label{c10}
\end{align}
\par We claim that for any $k \in \mathbb{N}_{+}$, the solution to \eqref{liearSys1}--\eqref{c7} exists in a uniform existence region $(t,x)\in\mathbb{R}\times[\tilde{x}-\delta,L]$ for small enough $\delta$. In fact, first, one can extend the boundary conditions periodically to $t\in \mathbb{R}$. Then, for linearized systems and decoupled boundary conditions, one can easily write the classical solution by method of characteristics. Additionally, the characteristics emanating from $\chi^{(k-1)}(t)$ can develop both forward and backward with respect to $x$.
\par Second, with such $\breve{\mathbf{\Upsilon}}^{(k)}$ and on the basis of \Cref{lemmaODE1} and \eqref{pFpx}, we construct the approximation shock position $\chi^{(k)}(t)$ as the $\mathcal{T}$--periodic solution of the ODE
\begin{align}\label{c11}
\frac{d\chi^{(k)}(t)}{dt}=\mathcal{F}^{(k)}\left[t,\chi^{(k)}(t)\right],
\end{align}
where we use the notation
\begin{multline}
\nonumber
\mathcal{F}^{(k)}\left[t,\chi^{(k)}(t)\right]\\
:= \mathcal{F}\left(\chi^{(k)}(t), \left(\breve{\Upsilon}_{1}^{(k)}, \frac{1}{\alpha_{2}}\breve{\Upsilon}_{2}^{(k)}, \frac{1}{\alpha_{3}}\breve{\Upsilon}_{3}^{(k)}\right)(t,\chi^{(k)}(t)), \bar{\rho}_{l}(t,\chi^{(k)}(t)), \bar{u}_{l}(t,\chi^{(k)}(t)), \bar{p}_{l}(t,\chi^{(k)}(t))\right).
\end{multline}
\par In the rest part of this section, we are devoted to prove the following proposition.
\begin{proposition}\label{propo}
	For the iteration scheme \eqref{liearSys1}--\eqref{c7} and \eqref{c11} with \eqref{boundarySmall}--\eqref{boundaryPeriodicity} and \eqref{c10}, under the assumptions~\eqref{a11}-\eqref{AA11}, there exist $\theta\in (0,1)$, small enough $\epsilon_1>0$ and big enough $C_{P}, C_{\mathcal{F}},C_{M}>0$ such that for any given $\epsilon\in(0,\epsilon_1)$, the approximate sequence $\breve{\mathbf{\Upsilon}}^{(k)}, \chi^{(k)}$ satisfies that for any $k\in\mathbb{N}_{+}$:
	\begin{align}
		&\breve{\mathbf{\Upsilon}}^{(k)}(t+\mathcal{T},x)=\breve{\mathbf{\Upsilon}}^{(k)}(t,x),
		~~\chi^{(k)}(t+\mathcal{T})=\chi^{(k)}(t), \quad \forall (t,x)\in \mathbb{R}\times[\tilde{x}-\delta,L],\label{appPeriodic}\\
		&\|\breve{\mathbf{\Upsilon}}^{(k)}\|_{C^{1}}:= \max_{i=1,2,3}\left\{\|\breve{\Upsilon}_{i}^{(k)}\|, \|\partial_{t}\breve{\Upsilon}_{i}^{(k)}\|, \|\partial_{x}\breve{\Upsilon}_{i}^{(k)}\| \right\}<C_{P}\epsilon,\label{C1estima}\\
		&\|\chi^{(k)}-\tilde{x}\|_{C^{1}} := \max\left\{\|\chi^{(k)}-\tilde{x}\|, \|{\chi^{(k)}}^{\prime}\|\right\}<C_{P}\epsilon\mathcal{E}^{-1},\label{appShockesti} \\
		&\|\breve{\mathbf{\Upsilon}}^{(k)}-\breve{\mathbf{\Upsilon}}^{(k-1)}\|\leq C_{P}\epsilon\theta^{k-1},\label{cauchyEstima1}\\
		&\|\chi^{(k)}-\chi^{(k-1)}\|_{C^{1}}\leq C_{\mathcal{F}}\epsilon\mathcal{E}^{-1}\theta^{k-1},\label{cauchyEstima2}
	\end{align}
	and
	\begin{align}
		\max\limits_{i=1,2,3}\left\{\varpi(\eta|\partial_{t}\breve{\Upsilon}_{i}^{(k)}(\cdot,\cdot))
		+\varpi(\eta|\partial_{x}\breve{\Upsilon}_{i}^{(k)}(\cdot,\cdot))\right\}<C_{M}\Phi(\eta). \label{modulus}
	\end{align}
	Here $\varpi$ is the modulus of continuity defined by
	\begin{eqnarray*}
		\varpi(\eta|h(\cdot,\cdot)):=\sup\limits_{\substack{|t_{1}-t_{2}|\leq\eta\\|x_{1}-x_{2}|\leq\eta }}
		\left|h(t_{1},x_{1})-h(t_{2},x_{2})\right|,
	\end{eqnarray*}
	for some function $h\in C_{t,x}^{0}$ and $\Phi(\eta)$ is a continuous function on $\eta \in (0,1)$, independent of $k$, to be determined later with
	\begin{eqnarray*}
		\lim\limits_{\eta\to 0^{+}}\Phi(\eta)=0.
	\end{eqnarray*}
\end{proposition}
Similarly to \cite{Qu}, once we show \Cref{propo}, one can directly see \Cref{t1}. In fact, first, by \eqref{cauchyEstima1} and \eqref{cauchyEstima2}, the sequence $\{\breve{\mathbf{\Upsilon}}^{(k)}\}_{k=1}^{\infty}$ and $\{\chi^{(k)}\}_{k=1}^{\infty}$ are Cauchy sequences in $C^{0}$ and $C^{1}$ spaces respectively, and thus converge to some $C^{0}$ function $\breve{\mathbf{\Upsilon}}^{(\mathcal{T})}$ and  $C^{1}$ function $\chi^{(\mathcal{T})}$, respectively. Additionally, by \eqref{appPeriodic}, the limits are $\mathcal{T}$--periodic. Moreover, by \eqref{C1estima} and \eqref{modulus}, applying the Arzel\`{a}--Ascoli theorem, there exists a subsequence of $\{\breve{\mathbf{\Upsilon}}^{(k)}\}$ that converges uniformly in $C^{1}$. By the uniqueness of the limit, one can conclude that $\{\breve{\mathbf{\Upsilon}}^{(k)}\}$ actually converges to $\breve{\mathbf{\Upsilon}}^{(\mathcal{T})}$ in $C^{1}$ space. Thus, with the shock position $\chi^{(\mathcal{T})}$, one can see that $\breve{\mathbf{\Upsilon}}^{(\mathcal{T})}|_{\Omega_{+}^{\mathcal{T}}}$ is a $C^{1}$ classical solution to the system \eqref{waveDe}--\eqref{boundaryPeriodicity} in the subsonic region defined by \eqref{regions}. Applying the variable substitution inversely, one obtains the solution of \eqref{a1} and \eqref{a5} in the subsonic region. As for the solution in the supersonic region, we take $(\rho_{-}^{(\mathcal{T})},u_{-}^{(\mathcal{T})},p_{-}^{(\mathcal{T})})|_{\Omega_{-}^{\mathcal{T}}}$, which is given by \Cref{existSuper}. Moreover, one can recover the R--H conditions, equivalently \eqref{RHcon} and \eqref{speedEx}, by passing to the limit in \eqref{c9} and \eqref{c11}, and the Lax entropy condition follows. Finally, we take the initial data just in the form of \eqref{a14}, and then the proof of \Cref{t1} is complete.
\par The proof of \Cref{propo} is quite lengthy. We will divide it into three subsections.
\subsection{Uniform $C^{1}$ boundedness and time--periodicity}\label{suu1}
\begin{proof}[Proof of \eqref{appPeriodic}--\eqref{appShockesti}]
\par For $k\ge 1$, with $\breve{\mathbf{\Upsilon}}^{(k-1)}$ and $\chi^{(k-1)}$ in hand, we will inductively prove the following
\begin{align}
	&\breve{\mathbf{\Upsilon}}^{(k)}(t+\mathcal{T},x)=\breve{\mathbf{\Upsilon}}^{(k)}(t,x),
	~~\chi^{(k)}(t+\mathcal{T})=\chi^{(k)}(t), \quad \forall (t,x)\in \mathbb{R}\times[\tilde{x}-\delta,L], \label{c17}\\
	&\|\breve{\mathbf{\Upsilon}}^{(k)}\|<C_{1}\epsilon,\quad \|\partial_{t}\breve{\mathbf{\Upsilon}}^{(k)}\|<C_{1}\epsilon, \quad \|\partial_{x}\breve{\mathbf{\Upsilon}}^{(k)}\|<C_{2}\epsilon,\label{c18}\\
	&\|\chi^{(k)}-\tilde{x}\|<C_{\mathcal{F},0}\epsilon\mathcal{E}^{-1},\quad \|{\chi^{(k)}}^{\prime}\|<C_{\mathcal{F},1}\epsilon,\label{c19}
\end{align}
under the assumption 
\begin{align}
	&\breve{\mathbf{\Upsilon}}^{(k-1)}(t+\mathcal{T},x)=\breve{\mathbf{\Upsilon}}^{(k-1)}(t,x),
	~~\chi^{(k-1)}(t+\mathcal{T})=\chi^{(k-1)}(t), \quad \forall (t,x)\in \mathbb{R}\times[\tilde{x}-\delta,L], \label{c20}\\
	&\|\breve{\mathbf{\Upsilon}}^{(k-1)}\|<C_{1}\epsilon,\quad \|\partial_{t}\breve{\mathbf{\Upsilon}}^{(k-1)}\|<C_{1}\epsilon, \quad \|\partial_{x}\breve{\mathbf{\Upsilon}}^{(k-1)}\|<C_{2}\epsilon,\label{c21}\\
	&\|\chi^{(k-1)}-\tilde{x}\|<C_{\mathcal{F},0}\epsilon\mathcal{E}^{-1},\quad \|{\chi^{(k-1)}}^{\prime}\|<C_{\mathcal{F},1}\epsilon,\label{c22}
\end{align}
where $C_{1},C_{2}, C_{\mathcal{F},0},C_{\mathcal{F},1}$ are positive constants to be determined later. Especially, in view of \eqref{c21}, we have $\breve{\mathbf{\Upsilon}}^{(k-1)} \in \mathcal{U}$ for some uniform small neighborhood $\mathcal{U}$, which guarantee that we only need to consider small perturbations. To start with, for $k=1$, \eqref{c20}--\eqref{c22} follow directly from \eqref{start}. 
\par First of all, by \eqref{c20}, one can easily see that $\breve{\mathbf{\Upsilon}}^{(k)}(t+\mathcal{T},x)$ solves equations \eqref{liearSys1}--\eqref{c7} if $\breve{\mathbf{\Upsilon}}^{(k)}(t,x)$ does. Thus, by the uniqueness to this linear system, one has 
\begin{eqnarray*}
	\breve{\mathbf{\Upsilon}}^{(k)}(t+\mathcal{T},x)=\breve{\mathbf{\Upsilon}}^{(k)}(t,x),
	\quad \forall (t,x)\in \mathbb{R}\times[\tilde{x}-\delta,L].
\end{eqnarray*}
\par Then, we present the estimates of the third wave $\breve{\Upsilon}_{3}^{(k)}$ in detail, which is the most essential part. For any fixed $(t,x) \in \mathbb{R}\times[\tilde{x}-\delta, L]$, one can define the 3--characteristics as
\begin{align}\label{characteritic}
\left\{
\begin{aligned}
&\frac{dX_{3}^{(k)}}{d\tau}(\tau;t,x)=\lambda_{3}\left(X_{3}^{(k)}(\tau;t,x),\breve{\mathbf{\Upsilon}}^{(k-1)}\left(\tau,X_{3}^{(k)}(\tau;t,x)\right)\right),\\
&X_{3}^{(k)}(t;t,x)=x,
\end{aligned}
\right.
\end{align}
and set the backward exit--time for $(t,x)$ as $\tau_{3}^{(k)}(t,x)$ by the intersection of the characteristic $X_{3}^{(k)}$ and the shock $\chi^{(k-1)}$, i.e.,
\begin{equation}\label{exitTime}
	X_{3}^{(k)}(\tau_{3}^{(k)};t,x)=\chi^{(k-1)}(\tau_{3}^{(k)}).
\end{equation}
Due to the different propagation speeds of the two, the intersection is uniquely well--defined. Then, one can integrate \eqref{liearSys3} along the characteristic curve $X_{3}^{(k)}=X_{3}^{(k)}(\tau;t,x)$ to get
\begin{multline}
	\breve{\Upsilon}_{3}^{(k)}(t,x)=\breve{\Upsilon}_{3}^{(k)}\left(\tau_{3}^{(k)},\chi^{(k-1)}(\tau_{3}^{(k)})\right) \\
	+\int_{\tau_{3}^{(k)}}^{0}\left(\alpha_{3}g_{3}(x, \breve{\mathbf{\Upsilon}}^{(k-1)}) - \sum_{i=1,2}\frac{\alpha_{3}l_{3i}^{(k-1)}}{\alpha_{i}}\left(\partial_{t}\breve{\Upsilon}_{i}^{(k-1)}+\lambda_{3}^{(k-1)}\partial_{x}\breve{\Upsilon}_{i}^{(k-1)}\right)\right)\left(\tau,X_{3}^{(k)}(\tau;t,x)\right)d\tau.
\end{multline}
By \eqref{a11}, \eqref{c9} and applying Hadamard's formula, one has
\begin{align}
\left|\breve{\Upsilon}_{3}^{(k)}(t,x)\right|
\leq&\left|\mathscr{A}_{3}^{(k-1)}\left[\tau_{3}^{(k)},\chi^{(k-1)}(\tau_{3}^{(k)})\right]\right|
+\alpha_{3}C_{L}\mathcal{E}\epsilon +CC_{1}\epsilon^{2} \notag\\
\leq& \alpha_{3}(1+C\epsilon +C\mathcal{E})\left(|\frac{\partial\mathscr{A}_{3}}{\partial x}(\tilde{x},\mathbf{0})|\|{\chi^{(k-1)}}-\tilde{x}\| +|\frac{\partial\mathscr{A}_{3}}{\partial \bar{\Upsilon}_{1}}(\tilde{x},\mathbf{0})|\|\bar{\Upsilon}_{1}^{(k-1)}\| \right.\notag\\
&\left. +|\frac{\partial\mathscr{A}_{3}}{\partial \bar{\rho}_{l}}(\tilde{x},\mathbf{0})|\|\bar{\rho}_{l}\| +|\frac{\partial\mathscr{A}_{3}}{\partial \bar{u}_{l}}(\tilde{x},\mathbf{0})|\|\bar{u}_{l}\| +|\frac{\partial\mathscr{A}_{3}}{\partial \bar{p}_{l}}(\tilde{x},\mathbf{0})|\|\bar{p}_{l}\|\right)
+C\mathcal{E}\epsilon +C\epsilon^{2}.\label{c36}
\end{align}
Here and below $C$ denotes different constants independent of $k$. Furthermore, applying \Cref{lemmaODE1}, one has
\begin{multline}\label{chiEsti}
	\|\chi^{(k-1)}-\tilde{x}\|\leq(1+C\epsilon)\frac{1}{|\frac{\partial \mathcal{F}}{\partial x}(\tilde{x},\mathbf{0})|}\left(\sum_{i=1}^{3}\dfrac{1}{\alpha_{i}}|\frac{\partial \mathcal{F}}{\partial \bar{\Upsilon}_{i}}(\tilde{x},\mathbf{0})|\|\breve{\Upsilon}_{i}^{(k-1)}\| \right. \\
	\left. +|\frac{\partial \mathcal{F}}{\partial \bar{\rho}_{l}}(\tilde{x},\mathbf{0})|\|\bar{\rho}_{l}\|
	+|\frac{\partial \mathcal{F}}{\partial \bar{u}_{l}}(\tilde{x},\mathbf{0})|\|\bar{u}_{l}\|+|\frac{\partial \mathcal{F}}{\partial \bar{p}_{l}}(\tilde{x},\mathbf{0})|\|\bar{p}_{l}\|\right).
\end{multline}
Therefore, applying the structural dissipation \eqref{dissipationCon}, one can get
\begin{equation}\label{C0estiFor3}
	\left|\breve{\Upsilon}_{3}^{(k)}(t,x)\right|
	\leq (1+C\epsilon +C\mathcal{E})\left(\alpha_{3}\theta_{d} + \frac{\alpha_{3}}{\alpha_{2}}\right)C_{1}\epsilon + CC_{l}\epsilon +C\mathcal{E}\epsilon +C\epsilon^{2}.
\end{equation}
Thus, since \eqref{beta}, for large enough $C_{1}$, one has
\begin{equation}
	\|\breve{\Upsilon}_{3}^{(k)}\| \le C_{1}\epsilon.
\end{equation}
The dissipation for $2$--waves at the shock and $1$--waves at the outlet boundary is guaranteed by scaling \eqref{sc} and $\alpha_{3} >1$ respectively. The analysis of them is similar or simpler, so we omit here and in subsequent part of this paper.
\par Next, we turn to the estimates for derivatives. We need to estimate the derivative in the region using the derivative along the boundary $\chi^{(k-1)}(t)$. To this end, we define function $\omega^{(k)}_{3}(t,x)$ as
\begin{equation}
	\omega_{3}^{(k)}\left(t,x\right)=\mathcal{F}^{(k-1)}\left(\tau_{3}^{(k)}(t,x),\chi^{(k-1)}(\tau_{3}^{(k)}(t,x))\right),
\end{equation}
which is well--defined following the definition of the backward exit--time. One can easily see that $\omega_{3}^{(k)}$ remains invariant along the characteristics, namely, it satisfies
\begin{equation}\label{omegaInvariant}
	\frac{\partial \omega_{3}^{(k)}}{\partial t} + \lambda_{3}^{(k-1)}\frac{\partial \omega_{3}^{(k)}}{\partial x} = 0.
\end{equation} Moreover, taking $t$--derivative to boundary condition \eqref{c9} yields
\begin{equation}\label{tDerivation}
	\frac{d}{dt}\breve{\Upsilon}_{3}^{(k)}(t,\chi^{(k-1)}(t)) = \left(\frac{\partial}{\partial t}+\mathcal{F}^{(k-1)}(t,\chi^{(k-1)}(t))\frac{\partial}{\partial x}\right)\breve{\Upsilon}_{3}^{(k)}(t,\chi^{(k-1)}(t)).
\end{equation}
Then, one can take the derivation operator 
\begin{equation}\label{dOmega}
	D_{\omega}^{(k)}:=\frac{\partial}{\partial t}+\omega_{3}^{(k)}(t,x)\frac{\partial}{\partial x}
\end{equation}
to the equation \eqref{liearSys3} and get
\begin{align}
	&\left(\partial_{t}+\lambda_{3}^{(k-1)}\partial_{x}\right)D_{\omega}^{(k)}\breve{\Upsilon}_{3}^{(k)}\notag\\
	=&-D_{3}^{(k)}\lambda_{3}^{(k-1)}\partial_{x}\breve{\Upsilon}_{3}^{(k)} + \alpha_{3}D_{3}^{(k)}g_{3} - \sum_{i=1,2}\frac{\alpha_{3}D_{3}^{(k)}l_{3i}^{(k-1)}}{\alpha_{i}}\left(\partial_{t}\breve{\Upsilon}_{i}^{(k-1)}+\lambda_{3}^{(k-1)}\partial_{x}\breve{\Upsilon}_{i}^{(k-1)}\right) \notag\\
	&-\sum_{i=1,2}\frac{\alpha_{3}l_{3i}^{(k-1)}}{\alpha_{i}}D_{3}^{(k)}\lambda_{3}^{(k-1)}\partial_{x}\breve{\Upsilon}_{i}^{(k-1)} - \sum_{i=1,2}\left(\partial_{t}+\lambda_{3}^{(k-1)}\partial_{x}\right)\left(\frac{\alpha_{3}l_{3i}^{(k-1)}}{\alpha_{i}}D_{\omega}^{(k)}\breve{\Upsilon}_{i}^{(k-1)}\right) \notag\\
	&+\sum_{i=1,2}\left(\left(\partial_{t}+\lambda_{3}^{(k-1)}\partial_{x}\right)\frac{\alpha_{3}l_{3i}^{(k-1)}}{\alpha_{i}}\right)\left(D_{\omega}^{(k)}\breve{\Upsilon}_{i}^{(k-1)}\right), \label{deriTrans}
\end{align}
where we use the operator notation for some function $f=f(x,\breve{\mathbf{\Upsilon}}^{(k-1)})$:
\begin{equation}\label{D3Forshort}
	D_{3}^{(k)} := \sum_{i=1}^{3}\left(\frac{\partial\breve{\Upsilon}_{i}^{(k-1)}}{\partial t}\frac{\partial}{\partial \breve{\Upsilon}_{i}} + \omega_{3}^{(k)}\frac{\partial\breve{\Upsilon}_{i}^{(k-1)}}{\partial x}\frac{\partial}{\partial \breve{\Upsilon}_{i}}\right) + \omega_{3}^{(k)}\frac{\partial}{\partial x}
\end{equation}
for simplification. Then, one can integrate \eqref{deriTrans} along the characteristic curve $X_{3}^{(k)}=X_{3}^{(k)}(\tau;t,x)$ and use \eqref{tDerivation}, \eqref{c21}--\eqref{c22} to get
\begin{align}
	&\left|D_{\omega}^{(k)}\breve{\Upsilon}_{3}^{(k)}(t,x)\right| \notag\\
	\le& \left|D_{\omega}^{(k)}\breve{\Upsilon}_{3}^{(k)}\left(\tau_{3}^{(k)},\chi^{(k-1)}(\tau_{3}^{(k)})\right)\right| +C\epsilon\|\partial_{x}\breve{\Upsilon}_{3}^{(k)}\|+C\epsilon^{2}+C\mathcal{E}\epsilon\notag\\
	\le& \alpha_{3}(1+C\epsilon +C\mathcal{E})\left(|\frac{\partial\mathscr{A}_{3}}{\partial x}(\tilde{x},\mathbf{0})|\|{\chi^{(k-1)}}^{\prime}\| +|\frac{\partial\mathscr{A}_{3}}{\partial \bar{\Upsilon}_{1}}(\tilde{x},\mathbf{0})|\left(\|\partial_{t}\bar{\Upsilon}_{1}^{(k-1)}\|+\|{\chi^{(k-1)}}^{\prime}\|\|\partial_{x}\bar{\Upsilon}_{1}^{(k-1)}\|\right) \right)\notag\\
	&+CC_{l}\epsilon +C\epsilon\|\partial_{x}\breve{\Upsilon}_{3}^{(k)}\|
	+C\mathcal{E}\epsilon +C\epsilon^{2}. \label{c42}
\end{align}
Again applying \Cref{lemmaODE1}, one can get
\begin{multline}\label{chiPrimeEsti}
	\|{\chi^{(k-1)}}^{\prime}\|\leq
	(2+C\epsilon)\left(\sum_{i=1}^{3}\dfrac{1}{\alpha_{i}}|\frac{\partial \mathcal{F}}{\partial \bar{\Upsilon}_{i}}(\tilde{x},\mathbf{0})|\|\breve{\Upsilon}_{i}^{(k-1)}\| \right. \\
	\left. +|\frac{\partial \mathcal{F}}{\partial \bar{\rho}_{l}}(\tilde{x},\mathbf{0})|\|\bar{\rho}_{l}\|
	+|\frac{\partial \mathcal{F}}{\partial \bar{u}_{l}}(\tilde{x},\mathbf{0})|\|\bar{u}_{l}\|+|\frac{\partial \mathcal{F}}{\partial \bar{p}_{l}}(\tilde{x},\mathbf{0})|\|\bar{p}_{l}\|\right).
\end{multline}
Thus, combining \eqref{pA3px},\eqref{c10}, \eqref{c21} and \eqref{c22}, one has
\begin{equation}\label{c43}
	\|\partial_{t}\breve{\Upsilon}_{3}^{(k)}\| - C\epsilon\|\partial_{x}\breve{\Upsilon}_{3}^{(k)}\| \le \alpha_{3}(1+C\epsilon +C\mathcal{E})\left(C\mathcal{E}\epsilon+\theta_{d}C_{1}\epsilon\right)+CC_{l}\epsilon +C\epsilon\|\partial_{x}\breve{\Upsilon}_{3}^{(k)}\|
	+C\mathcal{E}\epsilon +C\epsilon^{2}.
\end{equation}
Using the linear equation \eqref{liearSys3}, one can directly see
\begin{align}
\|\partial_{x}\breve{\Upsilon}_{3}^{(k)}\|\leq\mu_{\max}\|\partial_{t}\breve{\Upsilon}_{3}^{(k)}\|+C\mathcal{E}+C\epsilon,\label{c44}
\end{align}
where
\begin{eqnarray*}
	\mu_{\max}=\max\limits_{\substack{i=1,2,3\\\breve{\mathbf{\Upsilon}}\in \mathcal{U}}}\frac{1}{\left|\lambda_{i}\left(x,\breve{\mathbf{\Upsilon}}\right)\right|}.
\end{eqnarray*}
Combining \eqref{c43} and \eqref{c44}, one has
\begin{align}
&\|\partial_{t}\breve{\Upsilon}_{3}^{(k)}\|<C_{1}\epsilon,\quad \|\partial_{t}\breve{\Upsilon}_{3}^{(k)}\|<C_{2}\epsilon.\label{c46}
\end{align}
for big enough $C_{1}$ and $C_{2}$.
\par Finally, with \eqref{c18} in hand, one can see that \eqref{c19} follows just as in \eqref{chiEsti} and \eqref{chiPrimeEsti}, replacing the number with $k$. Furthermore, the $\mathcal{T}$--periodicity of $\chi^{(k)}$ is guaranteed by \Cref{lemmaODE1}. And so far, we have proved \eqref{c17}--\eqref{c19}, and thus \eqref{appPeriodic}--\eqref{appShockesti}.
\end{proof}

\subsection{Cauchy properties of the approximate sequences}\label{suu2}
\begin{proof}[Proof of \eqref{cauchyEstima1} and \eqref{cauchyEstima2}]
For $k \ge 2$, we will inductively prove
\begin{align}
	&\|\breve{\mathbf{\Upsilon}}^{(k)}-\breve{\mathbf{\Upsilon}}^{(k-1)}\|\leq C_{1}\epsilon\theta^{k-1}, \label{c52}\\
	&\|\chi^{(k)}-\chi^{(k-1)}\|\leq C_{\mathcal{F}}\epsilon\mathcal{E}^{-1}\theta^{k-1}, \quad \|{\chi^{(k)}}^{\prime}-{\chi^{(k-1)}}^{\prime}\| \leq C_{\mathcal{F}}\epsilon\theta^{k-1},\label{c52.1}
\end{align}
under the assumption
\begin{align}
&\|\breve{\mathbf{\Upsilon}}^{(k-1)}-\breve{\mathbf{\Upsilon}}^{(k-2)}\|\leq C_{1}\epsilon\theta^{k-2}, \label{c53}\\
&\|\chi^{(k-1)}-\chi^{(k-2)}\|\leq C_{\mathcal{F}}\epsilon\mathcal{E}^{-1}\theta^{k-2}, \quad \|{\chi^{(k)}}^{\prime}-{\chi^{(k-1)}}^{\prime}\|\leq C_{\mathcal{F}}\epsilon\theta^{k-2},\label{c53.1}
\end{align}
for some $\theta \in (0,1)$ and $C_{\mathcal{F}}>0$ to be determined later. To start with, for $k=2$, \eqref{c53}--\eqref{c53.1} follows directly from \eqref{start} and \eqref{c18}--\eqref{c19}.
\par For $k \ge 2$, the equation \eqref{liearSys3} yields
\begin{align}
	&\left(\partial_{t}+\lambda_{3}^{(k-1)}\partial_{x}\right)\left(\breve{\Upsilon}_{3}^{(k)}- \breve{\Upsilon}_{3}^{(k-1)}\right) \notag \\
	=& - \left(\lambda_{3}^{(k-1)}-\lambda_{3}^{(k-2)}\right)\partial_{x}\breve{\Upsilon}_{3}^{(k-1)} +\alpha_{3}\left(g_{3}(x,\breve{\mathbf{\Upsilon}}^{(k-1)})-g_{3}(x,\breve{\mathbf{\Upsilon}}^{(k-2)})\right) \notag \\
	& - \sum_{i=1,2}\frac{\alpha_{3}}{\alpha_{i}}\left(l_{3i}^{(k-1)}-l_{3i}^{(k-2)}\right)\left(\partial_{t}+\lambda_{3}^{(k-1)}\partial_{x}\right)\breve{\Upsilon}_{i}^{(k-1)} - \sum_{i=1,2}\frac{\alpha_{3}l_{3i}^{(k-2)}}{\alpha_{i}}\left(\lambda_{3}^{(k-1)}-\lambda_{3}^{(k-2)}\right)\partial_{x}\breve{\Upsilon}_{3}^{(k-2)} \notag \\
	& - \sum_{i=1,2}\left(\partial_{t}+\lambda_{3}^{(k-1)}\partial_{x}\right)\left(\frac{\alpha_{3}l_{3i}^{(k-2)}}{\alpha_{i}}\left(\breve{\Upsilon}_{i}^{(k-1)}- \breve{\Upsilon}_{i}^{(k-2)}\right)\right) \notag \\
	& + \sum_{i=1,2}\left(\breve{\Upsilon}_{i}^{(k-1)}- \breve{\Upsilon}_{i}^{(k-2)}\right)\left(\partial_{t}+\lambda_{3}^{(k-1)}\partial_{x}\right)\left(\frac{\alpha_{3}l_{3i}^{(k-2)}}{\alpha_{i}}\right). \label{c54}
\end{align}
Then for any fixed $(t,x) \in \mathbb{R}\times[\tilde{x}-\delta]$, one can integrate \eqref{c54} along the characteristic $X_{3}^{(k)}$ from $\tau_{3}^{(k)}(t,x)$ to $t$ and apply Hadamard's formula to get
\begin{align}
	&\left|\breve{\Upsilon}_{3}^{(k)}(t,x)-\breve{\Upsilon}_{3}^{(k-1)}(t,x)\right| \notag \\
	\le & \left|\breve{\Upsilon}_{3}^{(k)}\left(\tau_{3}^{(k)},\chi^{(k-1)}(\tau_{3}^{(k)})\right)
	-\breve{\Upsilon}_{3}^{(k-1)}\left(\tau_{3}^{(k)},\chi^{(k-2)}(\tau_{3}^{(k)})\right)\right|  \notag \\
	&+ \left|\breve{\Upsilon}_{3}^{(k-1)}\left(\tau_{3}^{(k)},\chi^{(k-2)}(\tau_{3}^{(k)})\right)
	-\breve{\Upsilon}_{3}^{(k-1)}\left(\tau_{3}^{(k)},\chi^{(k-1)}(\tau_{3}^{(k)})\right)\right| + \left(C\epsilon+C\mathcal{E}\right)\|\breve{\mathbf{\Upsilon}}^{(k-1)}-\breve{\mathbf{\Upsilon}}^{(k-2)}\| \notag \\
	\le & \alpha_{3}(1+C\epsilon +C\mathcal{E})\left(|\frac{\partial\mathscr{A}_{3}}{\partial x}(\tilde{x},\mathbf{0})|\|\chi^{(k-1)}-\chi^{(k-2)}\| +|\frac{\partial\mathscr{A}_{3}}{\partial \bar{\Upsilon}_{1}}(\tilde{x},\mathbf{0})|\|\bar{\Upsilon}_{1}^{(k-1)}-\bar{\Upsilon}_{1}^{(k-2)}\| \right)\notag\\
	&+ (CC_{l}+ C_{1})\epsilon\|\chi^{(k-1)}-\chi^{(k-2)}\| +\left(C\epsilon+C\mathcal{E}\right)\|\breve{\mathbf{\Upsilon}}^{(k-1)}-\breve{\mathbf{\Upsilon}}^{(k-2)}\|. \label{c55}
\end{align}
Applying \Cref{lemmaODE2} to equation \eqref{c11} with number $k-1$ and $k-2$, one can get
\begin{multline}\label{c56}
	\|\chi^{(k-1)}-\chi^{(k-2)}\|\leq(1+C\epsilon)\frac{\exp(|\frac{\partial \mathcal{F}}{\partial x}(\tilde{x},\mathbf{0})|\mathcal{T})}{|\frac{\partial \mathcal{F}}{\partial x}(\tilde{x},\mathbf{0})|}\left(\sum_{i=1}^{3}\dfrac{1}{\alpha_{i}}|\frac{\partial \mathcal{F}}{\partial \bar{\Upsilon}_{i}}(\tilde{x},\mathbf{0})|\|\breve{\Upsilon}_{i}^{(k-1)}-\breve{\Upsilon}_{i}^{(k-2)}\| \right. \\
	\left. +\left(|\frac{\partial \mathcal{F}}{\partial \bar{\rho}_{l}}(\tilde{x},\mathbf{0})|
	+|\frac{\partial \mathcal{F}}{\partial \bar{u}_{l}}(\tilde{x},\mathbf{0})|+|\frac{\partial \mathcal{F}}{\partial \bar{p}_{l}}(\tilde{x},\mathbf{0})|\right)C_{l}\epsilon\|\chi^{(k-1)}-\chi^{(k-2)}\|\right).
\end{multline}
By \eqref{Fxsmall}, one has
\begin{eqnarray*}
	\exp\left(\left|\frac{\partial \mathcal{F}}{\partial x}(\tilde{x},\mathbf{0})\right|\mathcal{T}\right) = 1+O(1)\mathcal{E},
\end{eqnarray*}
and thus, combining \eqref{c55}, \eqref{c56} and \eqref{dissipationCon},
\begin{multline}\label{cauchy}
	\left|\breve{\Upsilon}_{3}^{(k)}(t,x)-\breve{\Upsilon}_{3}^{(k-1)}(t,x)\right| \\
	\le (1+C\epsilon +C\mathcal{E} +\frac{C\epsilon}{\mathcal{E}})\left(\alpha_{3}\theta_{d} + \dfrac{\alpha_{3}}{\alpha_{2}}\right)\|\breve{\mathbf{\Upsilon}}^{(k-1)}-\breve{\mathbf{\Upsilon}}^{(k-2)}\| +\left(C\epsilon+C\mathcal{E}\right)\|\breve{\mathbf{\Upsilon}}^{(k-1)}-\breve{\mathbf{\Upsilon}}^{(k-2)}\|.
\end{multline}
For small enough $\epsilon$, one can find suitable $\theta$ such that
\begin{eqnarray*}
	(1+C\epsilon +C\mathcal{E}+\frac{C\epsilon}{\mathcal{E}})\left(\alpha_{3}\theta_{d} + \dfrac{\alpha_{3}}{\alpha_{2}}\right) + C\epsilon +C \mathcal{E} < \theta < 1.
\end{eqnarray*}
Then with \eqref{c53} in hand, one has
\begin{align}
\left|\breve{\Upsilon}_{3}^{(k)}(t,x)-\breve{\Upsilon}_{3}^{(k-1)}(t,x)\right|\leq C_{1}\epsilon\theta^{k-1}.\label{c61}
\end{align}
The analysis for $\breve{\Upsilon}_{1}^{(k)}$ and $\breve{\Upsilon}_{2}^{(k)}$ is similar or simpler. With \eqref{c52} in hand, \eqref{c52.1} follows for suitable big $C_{\mathcal{F}}$ by applying \Cref{lemmaODE2}. Here, we have done the proof of \eqref{cauchyEstima1}--\eqref{cauchyEstima2}.
\end{proof}

\subsection{Estimates of continuity modulus}\label{suu3}
\begin{proof}[Proof of \eqref{modulus}]
To estimate the modulus of continuity as in \eqref{modulus}, we will prove for $k\ge 1$,
\begin{align}
&\mathop{\max}\limits_{i=1,2,3}\varpi(\eta|\partial_{t}\breve{\Upsilon}_{i}^{(k)}(\cdot,x))<C_{M,1}\Phi(\eta), \quad \mathop{\max}\limits_{i=1,2,3}\varpi(\eta|\partial_{x}\breve{\Upsilon}_{i}^{(k)}(\cdot,x))<C_{M,2}\Phi(\eta),\label{c63}\\
&\max\limits_{i=1,2,3}\left\{\varpi(\eta|\partial_{t}\breve{\Upsilon}_{i}^{(k)}(\cdot,\cdot))
+\varpi(\eta|\partial_{x}\breve{\Upsilon}_{i}^{(k)}(\cdot,\cdot))\right\}<C_{M}\Phi(\eta),\label{c64}
\end{align}
under the assumption
\begin{align}
	&\mathop{\max}\limits_{i=1,2,3}\varpi(\eta|\partial_{t}\breve{\Upsilon}_{i}^{(k-1)}(\cdot,x))<C_{M,1}\Phi(\eta), \quad \mathop{\max}\limits_{i=1,2,3}\varpi(\eta|\partial_{x}\breve{\Upsilon}_{i}^{(k-1)}(\cdot,x))<C_{M,2}\Phi(\eta),\label{c65}\\
	&\max\limits_{i=1,2,3}\left\{\varpi(\eta|\partial_{t}\breve{\Upsilon}_{i}^{(k-1)}(\cdot,\cdot))
	+\varpi(\eta|\partial_{x}\breve{\Upsilon}_{i}^{(k-1)}(\cdot,\cdot))\right\}<C_{M}\Phi(\eta),\label{c66}
\end{align}
where
\begin{eqnarray*}
	\varpi(\eta|h(\cdot,x))=\mathop{\sup}\limits_{|t_{1}-t_{2}|\leq\eta}|h(t_{1},x)-h(t_{2},x)|,
\end{eqnarray*}
and $C_{M,1},C_{M,2}$ and $C_{M}$ are constants to be determined later. Moreover, we define the modulus function by
\begin{align}
\Phi(\eta)=&\eta+\varpi(\eta|\bar{\Upsilon}_{1,b+}^{\prime})+\varpi(\eta|\frac{\partial\bar{\rho}_{l}}{\partial t}(\cdot,\cdot))
+\varpi(\eta|\frac{\partial\bar{\rho}_{l}}{\partial x}(\cdot,\cdot))+\varpi(\eta|\frac{\partial\bar{u}_{l}}{\partial t}(\cdot,\cdot))\notag\\
&+\varpi(\eta|\frac{\partial\bar{u}_{l}}{\partial x}(\cdot,\cdot))+\varpi(\eta|\frac{\partial\bar{p}_{l}}{\partial t}(\cdot,\cdot))+\varpi(\eta|\frac{\partial\bar{p}_{l}}{\partial x}(\cdot,\cdot)),\label{c67}
\end{align}
where for $h \in C_{t}^{1}$
\begin{eqnarray*}
	\varpi(\eta|h) = \sup_{|t_1 - t_2| \leq \eta} \left|h(t_1) - h(t_2)\right|.
\end{eqnarray*}
Since \eqref{boundarySmall} and \eqref{c10}, one has $\lim_{\eta \to 0+} \Phi(\eta) = 0$. To start with, \eqref{start} yields \eqref{c65}--\eqref{c66} for $k=1$ directly. In the rest of this subsection, we will prove \eqref{c63}--\eqref{c64} in 4 steps.
\begin{description}[leftmargin=0cm]
	\item[Step 1. Estimates for two points on the shock boundary]  For given $t_{1},t_{2}$ with $\left|t_{1}-t_{2}\right| \le \eta$, by taking the temporal derivative on the boundary condition \eqref{c9}, one has
	\begin{align}
		&\frac{1}{\alpha_{3}}\left(\frac{d}{dt}\breve{\Upsilon}_{3}^{(k)}(t_{2},\chi^{(k-1)}(t_{2})) -  \frac{d}{dt}\breve{\Upsilon}_{3}^{(k)}(t_{1},\chi^{(k-1)}(t_{1}))\right) \notag \\
		=& \left({\chi^{(k-1)}}^{\prime}(t_{2})- {\chi^{(k-1)}}^{\prime}(t_{1})\right)\frac{\partial \mathscr{A}_{3}^{(k-1)}}{\partial x}\left[t_{2},\chi^{(k)}(t_{2})\right] \notag\\
		&+ {\chi^{(k-1)}}^{\prime}(t_{1})\left(\frac{\partial \mathscr{A}_{3}^{(k-1)}}{\partial x}\left[t_{2},\chi^{(k)}(t_{2})\right]-\frac{\partial \mathscr{A}_{3}^{(k-1)}}{\partial x}\left[t_{1},\chi^{(k)}(t_{1})\right]\right) \notag \\
		&+ \left(\frac{d}{d t}\bar{\Upsilon}_{1}^{(k-1)}(t_{2},\chi^{(k-1)}(t_{2}))-\frac{d}{d t}\bar{\Upsilon}_{1}^{(k-1)}(t_{1},\chi^{(k-1)}(t_{1}))\right)\frac{\partial \mathscr{A}_{3}^{(k-1)}}{\partial \bar{\Upsilon}_{1}}\left[t_{2},\chi^{(k)}(t_{2})\right] \notag \\
		&+ \frac{d}{d t}\bar{\Upsilon}_{1}^{(k-1)}(t_{1},\chi^{(k-1)}(t_{1}))\left(\frac{\partial \mathscr{A}_{3}^{(k-1)}}{\partial \bar{\Upsilon}_{1}}\left[t_{2},\chi^{(k)}(t_{2})\right] - \frac{\partial \mathscr{A}_{3}^{(k-1)}}{\partial \bar{\Upsilon}_{1}}\left[t_{1},\chi^{(k)}(t_{1})\right]\right) \notag \\
		& + \sum_{i=1}^{3}\left(\frac{d}{d t}\bar{U}_{i,l}^{(k-1)}(t_{2},\chi^{(k-1)}(t_{2}))-\frac{d}{d t}\bar{U}_{i,l}^{(k-1)}(t_{1},\chi^{(k-1)}(t_{1}))\right)\frac{\partial \mathscr{A}_{3}^{(k-1)}}{\partial \bar{U}_{i,l}}\left[t_{2},\chi^{(k)}(t_{2})\right] \notag \\
		&+ \sum_{i=1}^{3} \frac{d}{d t}\bar{U}_{i,l}^{(k-1)}(t_{1},\chi^{(k-1)}(t_{1}))\left(\frac{\partial \mathscr{A}_{3}^{(k-1)}}{\partial \bar{U}_{i,l}}\left[t_{2},\chi^{(k)}(t_{2})\right] - \frac{\partial \mathscr{A}_{3}^{(k-1)}}{\partial \bar{U}_{i,l}}\left[t_{1},\chi^{(k)}(t_{1})\right]\right),
	\end{align}
	where we denote $(\bar{U}_{1,l},\bar{U}_{2,l}, \bar{U}_{3,l}) =(\bar{\rho}_{l}, \bar{u}_{l},\bar{p}_{l})$ for simplification. By \eqref{bb21}, \eqref{C1estima}, \eqref{c19} and \eqref{c65}, one has
	\begin{multline}\label{modulusDissipation}
		\left|\frac{d}{dt}\breve{\Upsilon}_{3}^{(k)}(t_{2},\chi^{(k-1)}(t_{2})) -  \frac{d}{dt}\breve{\Upsilon}_{3}^{(k)}(t_{1},\chi^{(k-1)}(t_{1}))\right| \le \left(C\mathcal{E}C_{P}\epsilon\eta + CC_{P}\epsilon\eta\right)\mathcal{E}+CC_{\mathcal{F},1}\epsilon^{2}\mathcal{E}\eta \\
		+\alpha_{3}\theta_{d}\left(C_{M,1}+C_{\mathcal{F},1}\epsilon C_{M,2}\right)\Phi(\eta)+C(1+C_{l}C_{1}\epsilon^{2})\Phi(\eta).
	\end{multline}
	This equation reveals the dissipation structure of the continuity modulus at the shock boundary. It will help us estimate the modulus in the form
	\begin{equation}
		\left|D_{\omega}^{(k)}\breve{\Upsilon}_{3}^{(k)}(t_{2},x) - D_{\omega}^{(k)}\breve{\Upsilon}_{3}^{(k)}(t_{1},x)\right|
	\end{equation}
	with the derivation operator defined by \eqref{dOmega}.
	
	\item[Step 2. Uniform $\|\omega_{3}^{(k)}\|_{C^{1}}$ estimates] We need to estimate the $C^{1}$ norm of $\omega_{3}^{(k)} = \omega_{3}^{(k)}(t,x)$. Taking the derivation operator $D_{\omega}^{(k)}$ to \eqref{omegaInvariant} yields
	\begin{equation}
		\frac{\partial}{\partial t} D_{\omega}^{(k)}\omega_{3}^{(k)} + \lambda_{3}^{(k-1)}\frac{\partial}{\partial x} D_{\omega}^{(k)}\omega_{3}^{(k)} = -D_{3}^{(k)}\lambda_{3}^{(k-1)}\frac{\partial \omega_{3}^{(k)}}{\partial x}.
	\end{equation}
	Integrating it along the characteristic \eqref{characteritic} and using \eqref{bb21}, one can obtain
	\begin{align}
		&\left|D_{\omega}^{(k)}\omega_{3}^{(k)}(t,x)\right| \notag\\
		\le& \left|D_{\omega}^{(k)}\omega_{3}^{(k)}\left(\tau_{3}^{(k)},\chi^{(k-1)}(\tau_{3}^{(k)})\right)\right| +C_{\mathcal{F},1}\epsilon\|\partial_{x}\omega_{3}^{(k)}\| \notag\\
		\le& (1+C\epsilon +C\mathcal{E})\left(\sum_{i=1}^{3}\dfrac{1}{\alpha_{i}}|\frac{\partial \mathcal{F}}{\partial \bar{\Upsilon}_{i}}(\tilde{x},\mathbf{0})|\|\partial_{t}\breve{\Upsilon}_{i}^{(k-1)}\| \right) +CC_{l}\epsilon +C\epsilon\|\partial_{x}\breve{\Upsilon}_{3}^{(k)}\|
		+CC_{1}\mathcal{E}\epsilon. \label{c72}
	\end{align}
	Then, one can use the similar argument as \eqref{c42}--\eqref{c46} to get
	\begin{align}
		\|\frac{\partial\omega_{2}^{(k)}}{\partial x}\|\leq\mu_{\max}\|\frac{\partial\omega_{2}^{(k)}}{\partial t}\|, \label{c73}
	\end{align}
	and thus obtain from \eqref{c72} that
	\begin{equation}\label{omegaC1}
		\|\frac{\partial\omega_{3}^{(k)}}{\partial t}\|<C_{\mathcal{F},2}\epsilon,\quad \|\frac{\partial\omega_{3}^{(k)}}{\partial x}\|\leq\mu_{\max}C_{\mathcal{F},2}\epsilon
	\end{equation}
	for some suitable large constant $C_{\mathcal{F},2}$ independent of $k$.
	
	\item[Step 3. Estimates in the rigion by the method of characteristic] To simplify the analysis, we set $\mu_{i} = (\lambda_{i})^{-1}, \mu_{i}^{(k-1)} = (\lambda_{i}^{(k-1)})^{-1}$. For any fixed $(t,x)\in \mathbb{R}\times[\tilde{x}-\delta,L]$, we denote the inverse function of characteristics \eqref{characteritic} by $T_{3}^{(k)}(y;t,x)$, i.e.,
	\begin{equation}
		X_{3}^{(k)}\left(T_{3}^{(k)}(y;t,x);t,x\right) \equiv y.
	\end{equation}
	Obviously, it is well--defined by the implicit function theorem and driven by the ODE
	\begin{align}\label{characteriticT}
		\left\{
		\begin{aligned}
			&\frac{dT_{3}^{(k)}}{dy}(y;t,x)=\mu_{3}\left(y,\breve{\mathbf{\Upsilon}}^{(k-1)}\left(T_{3}^{(k)}(y;t,x),y\right)\right),\\
			&T_{3}^{(k)}(x;t,x)=t.
		\end{aligned}
		\right.
	\end{align}
	\par Given two points $(t_{1},x), (t_{2},x)$ in the region with $|t_{1}-t_{2}| \le \eta$, for any $y\in [\tilde{x}-\delta,L]$, one can integrate \eqref{characteriticT} from $x$ to $y$ and get
	\begin{align}
		&\left|T_{3}^{(k)}(y;t_{2},x) - T_{3}^{(k)}(y;t_{1},x)\right| \notag \\
		\le &|t_{2} - t_{1}| + |\int_{x}^{y} \mu_{3}\left(z,\breve{\mathbf{\Upsilon}}^{(k-1)}\left(T_{3}^{(k)}(z;t_{2},x),z\right)\right)- \mu_{3}\left(z,\breve{\mathbf{\Upsilon}}^{(k-1)}\left(T_{3}^{(k)}(z;t_{1},x),z\right)\right)dz| \notag \\
		\le & |t_{2} - t_{1}| + CC_{P}\epsilon\int_{x}^{y} |T_{3}^{(k)}(z;t_{2},x) - T_{3}^{(k)}(z;t_{1},x)|dz.
	\end{align}
	Combining with Gronwall's inequality, one has
	\begin{equation}\label{T1T2esitm}
		\left|T_{3}^{(k)}(y;t_{2},x) - T_{3}^{(k)}(y;t_{1},x)\right| \le (1+CC_{P}\epsilon)|t_{2} - t_{1}| \le (1+CC_{P}\epsilon)\eta, \quad \forall y\in [\tilde{x}-\delta,L].
	\end{equation}
	Recalling the exit--time $\tau_{3}^{(k)}(t,x)$, we denote exit--position by
	\begin{eqnarray*}
		\xi_{3}^{(k)}(t,x) := X_{3}^{(k)}(\tau_{3}^{(k)};t,x),
	\end{eqnarray*}
	and by \eqref{exitTime}, one has
	\begin{equation}\label{xiChi}
		\xi_{3}^{(k)}(t,x) = \chi^{(k-1)}(T_{3}^{(k)}(\xi_{3}^{(k)};t,x)).
	\end{equation}
	Thus, one can see
	\begin{align}
		&\left|\xi_{3}^{(k)}(t_{1},x)- \xi_{3}^{(k)}(t_{2},x)\right| = \left|\chi^{(k-1)}(T_{3}^{(k)}(\xi_{3}^{(k)}(t_{1},x);t_{1},x))-\chi^{(k-1)}(T_{3}^{(k)}(\xi_{3}^{(k)}(t_{2},x);t_{2},x))\right| \notag \\
		\le & \|{\chi^{(k-1)}}^{\prime}\|\left(|T_{3}^{(k)}(\xi_{3}^{(k)}(t_{1},x);t_{1},x)-T_{3}^{(k)}(\xi_{3}^{(k)}(t_{2},x);t_{1},x)|\right.\notag\\
		&\qquad\qquad\qquad\qquad\qquad\qquad \left.-|T_{3}^{(k)}(\xi_{3}^{(k)}(t_{2},x);t_{1},x)-T_{3}^{(k)}(\xi_{3}^{(k)}(t_{2},x);t_{2},x)|\right) \notag \\
		\le & \|{\chi^{(k-1)}}^{\prime}\|\left(\mu_{\max}\left|\xi_{3}^{(k)}(t_{1},x)- \xi_{3}^{(k)}(t_{2},x)\right|+(1+CC_{P}\epsilon)\eta\right). \notag
	\end{align}
	Consequently, one has
	\begin{equation}\label{X1X2estim}
		\left|\xi_{3}^{(k)}(t_{1},x)- \xi_{3}^{(k)}(t_{2},x)\right| \le (1+CC_{\mathcal{F},1}\epsilon)C_{\mathcal{F},1}\eta.
	\end{equation}
	Then, one can integrate \eqref{deriTrans} along $T_{3}^{(k)}(y;t_{1},x)$ and $T_{3}^{(k)}(y;t_{2},x)$ respectively, and get
	\begin{align}
		& D_{\omega}^{(k)}\breve{\Upsilon}_{3}^{(k)}(t_{2},x) - D_{\omega}^{(k)}\breve{\Upsilon}_{3}^{(k)}(t_{1},x) \notag \\
		=&D_{\omega}^{(k)}\breve{\Upsilon}_{3}^{(k)}(\tau_{3}^{(k)}(t_{2},x), \xi_{3}^{(k)}(t_{2},x)) -D_{\omega}^{(k)}\breve{\Upsilon}_{3}^{(k)}(\tau_{3}^{(k)}(t_{1},x), \xi_{3}^{(k)}(t_{1},x)) \notag \\
		&-\int_{\xi_{3}^{(k)}(t_{2},x)}^{x} D_{3}^{(k)}\lambda_{3}^{(k-1)}\partial_{x}\breve{\Upsilon}_{3}^{(k)} \Big|_{(T_{3}^{(k)}(y;t_{1},x),y)}^{(T_{3}^{(k)}(y;t_{2},x),y)}dy \notag\\
		&- \int_{\xi_{3}^{(k)}(t_{1},x)}^{\xi_{3}^{(k)}(t_{2},x)} D_{3}^{(k)}\lambda_{3}^{(k-1)}\partial_{x}\breve{\Upsilon}_{3}^{(k)} (T_{3}^{(k)}(y;t_{1},x),y)dy \notag \\
		& - \sum_{i=1,2}\left(\frac{\alpha_{3}l_{3i}^{(k-1)}}{\alpha_{i}}D_{\omega}^{(k)}\breve{\Upsilon}_{i}^{(k-1)}\right)\Big|_{(t_{1},x)}^{(t_{2},x)}- \sum_{i=1,2}\left(\frac{\alpha_{3}l_{3i}^{(k-1)}}{\alpha_{i}}D_{\omega}^{(k)}\breve{\Upsilon}_{i}^{(k-1)}\right)\Big|_{(T_{3}^{(k)}(y;t_{1},x),y)}^{(T_{3}^{(k)}(y;t_{2},x),y)} \notag \\
		& +\int_{\xi_{3}^{(k)}(t_{2},x)}^{x} \Lambda^{(k)} \Big|_{(T_{3}^{(k)}(y;t_{1},x),y)}^{(T_{3}^{(k)}(y;t_{2},x),y)}dy + \int_{\xi_{3}^{(k)}(t_{1},x)}^{\xi_{3}^{(k)}(t_{2},x)} \Lambda^{(k)} (T_{3}^{(k)}(y;t_{1},x),y)dy, \label{modulusInt}
	\end{align}
	where we use the notation for simplification
	\begin{eqnarray*}
		f\big|_{(t_{1},x_{1})}^{(t_{2},x_{2})} = f(t_{2},x_{2}) - f(t_{1},x_{1}),
	\end{eqnarray*}
	for some function $f=f(t,x)$, and
	\begin{multline}
		\nonumber
		\Lambda^{(k)} := \alpha_{3}D_{3}^{(k)}g_{3} - \sum_{i=1,2}\frac{\alpha_{3}D_{3}^{(k)}l_{3i}^{(k-1)}}{\alpha_{i}}\left(\partial_{t}\breve{\Upsilon}_{i}^{(k-1)}+\lambda_{3}^{(k-1)}\partial_{x}\breve{\Upsilon}_{i}^{(k-1)}\right) \\
		 -\sum_{i=1,2}\frac{\alpha_{3}l_{3i}^{(k-1)}}{\alpha_{i}}D_{3}^{(k)}\lambda_{3}^{(k-1)}\partial_{x}\breve{\Upsilon}_{i}^{(k-1)}+\sum_{i=1,2}\left(\left(\partial_{t}+\lambda_{3}^{(k-1)}\partial_{x}\right)\frac{\alpha_{3}l_{3i}^{(k-1)}}{\alpha_{i}}\right)\left(D_{\omega}^{(k)}\breve{\Upsilon}_{i}^{(k-1)}\right).
	\end{multline}
	By the definition \eqref{dOmega}, one can see that
	\begin{multline}\label{triangleIneq}
		\left|\partial_{t}\breve{\Upsilon}_{3}^{(k)}(t_{2},x) - \partial_{t}\breve{\Upsilon}_{3}^{(k)}(t_{1},x) \right| \le \left|D_{\omega}^{(k)}\breve{\Upsilon}_{3}^{(k)}(t_{2},x) - D_{\omega}^{(k)}\breve{\Upsilon}_{3}^{(k)}(t_{1},x)\right| \\
		+ \left|\omega_{3}^{(k)}(t_{2},x)\right|\left|\partial_{x}\breve{\Upsilon}_{3}^{(k)}(t_{2},x) - \partial_{x}\breve{\Upsilon}_{3}^{(k)}(t_{1},x) \right| + \left|\omega_{3}^{(k)}(t_{2},x) - \omega_{3}^{(k)}(t_{1},x)\right|\left|\partial_{x}\breve{\Upsilon}_{3}^{(k)}(t_{1},x)\right|.
	\end{multline}
	Here, combining \eqref{modulusDissipation}, \eqref{omegaC1}, \eqref{T1T2esitm} and \eqref{X1X2estim}--\eqref{triangleIneq}, as well as $C^{1}$ estimates \eqref{C1estima}, one can get that for big enough $C_{M,1}>0$,
	\begin{equation}\label{modulusEstiX}
		\varpi(\eta|\frac{\partial \breve{\Upsilon}_{3}^{(k)}}{\partial t}(\cdot,x)) \le \theta_{0}C_{M,1}\Phi(\eta) + C\epsilon\varpi(\eta|\frac{\partial \breve{\Upsilon}_{3}^{(k)}}{\partial x}(\cdot,x)),
	\end{equation}
	where $\theta_{0}$ is some constant satisfying $\alpha_{3}\theta_{d} < \theta_{0} <1$.
	\par To end with, using the equation \eqref{liearSys3} and noting the known estimates \eqref{C1estima} and \eqref{c66}, one can derive
	\begin{equation}\label{modulusXfromT}
		\varpi(\eta|\frac{\partial \breve{\Upsilon}_{3}^{(k)}}{\partial x}(\cdot,x))
		\leq \mu_{\max}\varpi(\eta|\frac{\partial \breve{\Upsilon}_{3}^{(k)}}{\partial t}(\cdot,x))+(C\epsilon+C\mathcal{E})\Phi(\eta),
	\end{equation}
	which, together with \eqref{modulusEstiX}, completes the proof of \eqref{c63} for suitable big $C_{M,1}$, $C_{M,2}$.
	
	\item[Step 4. Uniform continuity modulus estimates] To complete this subsection, now we turn to the proof of \eqref{c64}. First, we consider two points $(t_{1},x_{1})$, $(t_{2},x_{2})$ with $|t_{1}-t_{2}|\le \eta$ and $|x_{1}-x_{2}|\le \eta$ located on a same characteristic $T_{3}^{(k)}(y;t_{1},x_{1})$, namely,
	\begin{eqnarray*}
		t_{2} = T_{3}^{(k)}(x_{2};t_{1},x_{1}).
	\end{eqnarray*}
	For this case, one can integrate the equation \eqref{deriTrans} along the characteristic to get
	\begin{align}
		&D_{\omega}^{(k)}\breve{\Upsilon}_{3}^{(k)}(t_{2},x_{2}) - D_{\omega}^{(k)}\breve{\Upsilon}_{3}^{(k)}(t_{1},x_{1}) \notag \\
		=&  \int_{x_{1}}^{x_{2}}\left(-D_{3}^{(k)}\lambda_{3}^{(k-1)}\partial_{x}\breve{\Upsilon}_{3}^{(k)}+\Lambda^{(k)}\right)(T_{3}^{(k)}(y;t_{1}, x_{1}),y)dy \notag\\
		&- \sum_{i=1,2}\left(\frac{\alpha_{3}l_{3i}^{(k-1)}}{\alpha_{i}}\right)\Big|_{(t_{1},x_{1})}^{(t_{2},x_{2})}D_{\omega}^{(k)}\breve{\Upsilon}_{i}^{(k-1)}(t_{2},x_{2}) - \sum_{i=1,2}\left(\frac{\alpha_{3}l_{3i}^{(k-1)}}{\alpha_{i}}\right)(t_{1},x_{1})D_{\omega}^{(k)}\breve{\Upsilon}_{i}^{(k-1)}\Big|_{(t_{1},x_{1})}^{(t_{2},x_{2})}.
	\end{align}
	And thus, with \eqref{Fxsmall}, \eqref{C1estima} and \eqref{c66}, one has
	\begin{equation}\label{sameCharacter}
		\left|D_{\omega}^{(k)}\breve{\Upsilon}_{3}^{(k)}(t_{2},x_{2}) - D_{\omega}^{(k)}\breve{\Upsilon}_{3}^{(k)}(t_{1},x_{1})\right| \le CC_{M}(\epsilon+\mathcal{E})\Phi(\eta).
	\end{equation}
	\par Finally, for general two points $(t_{1},x_{1})$, $(t_{2},x_{2})$ with $|t_{1}-t_{2}|\le \eta$ and $|x_{1}-x_{2}|\le \eta$, one can find a point $(t_{3},x_{2})$ located on the characteristic $T_{3}^{(k)}(y;t_{1},x_{1})$. By \eqref{characteriticT}, one has
	\begin{equation}
		\nonumber
		\left|t_{3}-t_{1}\right| \le \mu_{\max}\left|x_{2}-x_{1}\right| \le \mu_{\max}\eta,
	\end{equation}
	and thus
	\begin{equation}\label{muandone}
		\left|t_{3}-t_{2}\right| \le  \left(\mu_{\max}+1\right)\eta.
	\end{equation}
	Therefore, by \eqref{c63}, \eqref{sameCharacter} and \eqref{muandone}, one has
	\begin{align}
		&\left|D_{\omega}^{(k)}\breve{\Upsilon}_{3}^{(k)}(t_{2},x_{2}) - D_{\omega}^{(k)}\breve{\Upsilon}_{3}^{(k)}(t_{1},x_{1})\right| \notag\\
		\le & \left|D_{\omega}^{(k)}\breve{\Upsilon}_{3}^{(k)}(t_{2},x_{2}) - D_{\omega}^{(k)}\breve{\Upsilon}_{3}^{(k)}(t_{3},x_{2})\right| + \left|D_{\omega}^{(k)}\breve{\Upsilon}_{3}^{(k)}(t_{3},x_{2}) - D_{\omega}^{(k)}\breve{\Upsilon}_{3}^{(k)}(t_{1},x_{1})\right| \notag \\
		\le &\left(\mu_{\max}+2\right)C_{M,1}\Phi(\eta) + CC_{M}(\epsilon+\mathcal{E})\Phi(\eta).
	\end{align}
	Using the similar argument as \eqref{triangleIneq} and \eqref{modulusXfromT}, one can conclude
	\begin{equation}
		\varpi(\eta|\frac{\partial \breve{\Upsilon}_{3}^{(k)}}{\partial t}(\cdot,\cdot)) \le \left(\mu_{\max}+3\right)C_{M,1}\Phi(\eta),\quad \varpi(\eta|\frac{\partial \breve{\Upsilon}_{3}^{(k)}}{\partial x}(\cdot,\cdot)) \le \mu_{\max}\left(\mu_{\max}+3\right)C_{M,1}\Phi(\eta).
	\end{equation}
	One can choose $C_{M} = \left(\mu_{\max}+2\right)^{2}C_{M,1}$, and then the proof of \eqref{c64} is completed.
\end{description}
\end{proof}

\section{Stability of Time--periodic Transonic Shock Solutions}\label{s4}
In this section, we prove our main \Cref{t2}. To begin with, building upon the framework in~\cite{Xin, Rauch}, the initial--boundary value problem~\eqref{a1},\eqref{a3}--\eqref{a5} with all the assumptions \eqref{a8}--\eqref{a9},\eqref{compatibCri}, \eqref{a11}--\eqref{AA11} and initial data \eqref{a19}--\eqref{a20} admits a unique transonic shock solution $(\rho, u, p)(t,x)$ with shock position $\chi(t)$ satisfying
\begin{align}
\sum_{m=0}^{2}|\partial_{t}^{m}(\chi(t)-\tilde{x})|+\|(\bar{\rho}_{-},\bar{u}_{-},\bar{p}_{-})(t,\cdot)
\|_{C^{1}([0,\chi(t)))}
&+\|\breve{\mathbf{\Upsilon}}(t,\cdot)\|_{C^{1}((\chi(t),L])}
<C\epsilon,\label{d1}
\end{align}
where $$\bar{\rho}_{-}(t,x)=\rho_{-}(t,x)-\tilde{\rho}_{-}(x),~\bar{u}_{-}(t,x)=u_{-}(t,x)-\tilde{u}_{-}(x),
~\bar{p}_{-}(t,x)=p_{-}(t,x)-\tilde{p}_{-}(x).$$
\par Moreover, we denote the maximum time for the wave to reach the boundary by
\begin{equation}\label{timeMax}
	\mathcal{T}_{0}=\mu_{\max}L.
\end{equation}
Then in the supersonic region, the exponential decay follows via the method analogous to~\cite{Qup1}, namely, for $t\in[N\mathcal{T}_{0},(N+1)\mathcal{T}_{0}],~N\in\mathbb{Z}_{+}$,
\begin{align}
\|(\bar{\rho}_{-},\bar{u}_{-},\bar{p}_{-})(t,\cdot)-(\bar{\rho}_{-}^{(\mathcal{T})},\bar{u}_{-}^{(\mathcal{T})},
\bar{p}_{-}^{(\mathcal{T})})(t,\cdot)\|
_{C^{0}([0,\min\{\chi(t),\chi^{(\mathcal{T})}(t)\}))}\leq C_{L}\epsilon\zeta^{N},\label{d2}
\end{align}
where $C_{L}$ is a positive constant.
\par For the problem in the subsonic region, we denote by 
\begin{eqnarray*}
	\breve{\mathbf{\Upsilon}}^{(\mathcal{T})}(t,x) = \left(\breve{\Upsilon}_{1}^{(\mathcal{T})}, \breve{\Upsilon}_{2}^{(\mathcal{T})}, \breve{\Upsilon}_{3}^{(\mathcal{T})}\right)(t,x), \quad
	\breve{\mathbf{\Upsilon}}(t,x) = \left(\breve{\Upsilon}_{1}, \breve{\Upsilon}_{2}, \breve{\Upsilon}_{3}\right)(t,x),
\end{eqnarray*}
the solution to system \eqref{a1},\eqref{a3}--\eqref{a5} with the initial data \eqref{a14} and \eqref{a19} respectively in the associated regions
\begin{eqnarray*}
	\Omega_{+}^{\mathcal{T}} = \left\{(t,x):t\in \mathbb{R}_{+}, \chi^{(\mathcal{T})}(t)\le x \le L \right\}, \quad \Omega_{+} := \left\{(t,x):t\in \mathbb{R}_{+}, \chi(t)\le x \le L \right\}.
\end{eqnarray*}
In order to prove \eqref{a21}, with \eqref{d2} in hand, it remains to prove inductively that,
for some $t_{0}>0$ and $N \in \mathbb{Z}_{+}$, 
\begin{align}
	&\|\breve{\mathbf{\Upsilon}}(t,\cdot)-\breve{\mathbf{\Upsilon}}^{(\mathcal{T})}(t,\cdot)\| \le C_{S,1}\epsilon\zeta^{N+1}, \quad \forall t \in [t_{0}+\mathcal{T}_{0},t_{0}+2\mathcal{T}_{0}], \\
	&|\chi(t)
	-\chi^{(\mathcal{T})}(t)| \le C_{S,2}\epsilon\zeta^{N+1}, \quad |\chi^{\prime}(t)
	-{\chi^{(\mathcal{T})}}^{\prime}(t)| \le C_{S,2}\epsilon\zeta^{N+1}, \quad \forall t \in [t_{0}+\mathcal{T}_{0},t_{0}+2\mathcal{T}_{0}],
\end{align}
under the assumptions
\begin{align}
	&\|\breve{\mathbf{\Upsilon}}(t,\cdot)-\breve{\mathbf{\Upsilon}}^{(\mathcal{T})}(t,\cdot)\| \le C_{S,1}\epsilon\zeta^{N}, \quad \forall t \in [t_{0},t_{0}+\mathcal{T}_{0}], \\
	&|\chi(t)
	-\chi^{(\mathcal{T})}(t)| \le C_{S,2}\epsilon\zeta^{N}, \quad |\chi^{\prime}(t)
	-{\chi^{(\mathcal{T})}}^{\prime}(t)| \le C_{S,2}\epsilon\zeta^{N}, \quad \forall t \in [t_{0},t_{0}+\mathcal{T}_{0}].
\end{align}
To this end, we use a bootstrap argument as follows. Denote
\begin{align}
	&\Gamma_{1}(t)=\mathop{\max}\limits_{i=1,2,3}\mathop{\sup}\limits_{(t,x)\in\Omega_{+}^{\mathcal{T}}\cap\Omega_{+}}
	|\breve{\Upsilon}_{i}(t,x)
	-\breve{\Upsilon}^{(\mathcal{T})}_{i}(t,x)|, \notag \\
	&\Gamma_{2}(t)=
	|\chi(t)
	-\chi^{(\mathcal{T})}(t)|, \qquad \Gamma_{3}(t)=
	|\chi^{\prime}(t)
	-{\chi^{(\mathcal{T})}}^{\prime}(t)|. \notag
\end{align}
By the regularity of $\breve{\mathbf{\Upsilon}}$ and $\chi$, one can see that $\Gamma_{1}(t)$, $\Gamma_{2}(t)$ and $\Gamma_{3}(t)$ are continuous and satisfy
\begin{eqnarray*}
	\Gamma_{1}(t_{0}+\mathcal{T}_{0}) \le C_{S,1}\epsilon\zeta^{N}, \quad \Gamma_{2}(t_{0}+\mathcal{T}_{0}) \le C_{S,2}\epsilon\zeta^{N}, \quad \Gamma_{3}(t_{0}+\mathcal{T}_{0}) \le C_{S,2}\epsilon\zeta^{N}.
\end{eqnarray*}
Thus, it is sufficient to show  that for each given $t_* \in [t_{0}+\mathcal{T}_0, t_{0}+2\mathcal{T}_0]$, 
\begin{equation}\label{bootstrapGoal}
	\Gamma_{1}(t) \le C_{S,1}\epsilon\zeta^{N+1}, \quad \Gamma_{2}(t) \le C_{S,2}\epsilon\zeta^{N+1}, \quad \Gamma_{3}(t) \le C_{S,2}\epsilon\zeta^{N+1}, \quad \forall t\in [t_{0}+\mathcal{T}_{0}, t_{*}],
\end{equation}
under the assumptions
\begin{equation}\label{bootstrapAssum}
	\Gamma_{1}(t) \le \alpha C_{S,1}\epsilon\zeta^{N}, \quad \Gamma_{2}(t) \le  \alpha C_{S,2}\epsilon\zeta^{N}, \quad \Gamma_{3}(t) \le \alpha C_{S,2}\epsilon\zeta^{N}, \quad \forall t\in [t_{0}, t_{*}],
\end{equation}
for some constant $\alpha>1$ to be determined later.
\par Now, we use the equations \eqref{system3} for $\breve{\Upsilon}_{3}$ and $\breve{\Upsilon}_{3}^{(\mathcal{T})}$ respectively to get
\begin{align}
	&\left(\partial_{t}+\lambda_{3}\partial_{x}\right)\left(\breve{\Upsilon}_{3}- \breve{\Upsilon}_{3}^{(\mathcal{T})}\right) \notag \\
	=& - \left(\lambda_{3}-\lambda_{3}^{(\mathcal{T})}\right)\partial_{x}\breve{\Upsilon}_{3} +\alpha_{3}\left(g_{3}(x,\breve{\mathbf{\Upsilon}})-g_{3}(x,\breve{\mathbf{\Upsilon}}^{(\mathcal{T})})\right) \notag \\
	& - \sum_{i=1,2}\frac{\alpha_{3}}{\alpha_{i}}\left(l_{3i}-l_{3i}^{(\mathcal{T})}\right)\left(\partial_{t}+\lambda_{3}\partial_{x}\right)\breve{\Upsilon}_{i} - \sum_{i=1,2}\frac{\alpha_{3}l_{3i}^{(\mathcal{T})}}{\alpha_{i}}\left(\lambda_{3}-\lambda_{3}^{(\mathcal{T})}\right)\partial_{x}\breve{\Upsilon}_{3}^{(\mathcal{T})} \notag \\
	& - \sum_{i=1,2}\left(\partial_{t}+\lambda_{3}\partial_{x}\right)\left(\frac{\alpha_{3}l_{3i}^{(\mathcal{T})}}{\alpha_{i}}\left(\breve{\Upsilon}_{i}- \breve{\Upsilon}_{i}^{(\mathcal{T})}\right)\right) + \sum_{i=1,2}\left(\breve{\Upsilon}_{i}- \breve{\Upsilon}_{i}^{(\mathcal{T})}\right)\left(\partial_{t}+\lambda_{3}\partial_{x}\right)\left(\frac{\alpha_{3}l_{3i}^{(\mathcal{T})}}{\alpha_{i}}\right). \label{d10}
\end{align}
\par Next, one can define the characteristic as
\begin{align}\label{characteritic0}
	\left\{
	\begin{aligned}
		&\frac{dX_{3}}{d\tau}(\tau;t,x)=\lambda_{3}\left(X_{3}(\tau;t,x),\breve{\mathbf{\Upsilon}}\left(\tau,X_{3}(\tau;t,x)\right)\right),\\
		&X_{3}(t;t,x)=x,
	\end{aligned}
	\right.
\end{align}
for any fixed $(t,x) \in [t_{0}+\mathcal{T}_{0},t_{*}]\times[0,L]\cap\Omega_{+}^{\mathcal{T}}\cap\Omega_{+}$. Since the solution is $C^{1}$ smooth, the characteristic is well--defined. At the same way, one can define the characteristic $X_{3}^{(\mathcal{T})}$ for the time--periodic solution. In addition, one can set the backward exit--time for $(t,x)$ as respectively $\tau_{3}(t,x)$/$\tau_{3}^{(\mathcal{T})}(t,x)$ by the intersection of the characteristic $X_{3}$/$X_{3}^{(\mathcal{T})}$ and the shock $\chi$/$\chi^{(\mathcal{T})}$, i.e.,
\begin{equation}\label{exitTime0}
	X_{3}(\tau_{3};t,x)=\chi(\tau_{3}), \quad X_{3}^{(\mathcal{T})}(\tau_{3}^{(\mathcal{T})};t,x)=\chi^{(\mathcal{T})}(\tau_{3}^{(\mathcal{T})}).
\end{equation}
By \eqref{timeMax}, one has $\{\tau_{3}, \tau_{3}^{(\mathcal{T})}\} \subset [t_{0}, t_{*}]$.
\par Furthermore, one can integrate the ODE \eqref{shockposition} from $t_{0}$ to $t$ and use Hadamard's formula as well as \eqref{pFpx}--\eqref{Fxsmall} to get
\begin{multline}\label{smallZeta}
	\left(1+2\bar{C}\mathcal{E}\mathcal{T}_{0}\right)\Gamma_{2}(t) \le \alpha C_{S,2}\epsilon\zeta^{N} 
	+2\mathcal{T}_{0}(1+C\epsilon+C\mathcal{E})\left(\sum_{i=1}^{3}\dfrac{1}{\alpha_{i}}|\frac{\partial \mathcal{F}}{\partial \bar{\Upsilon}_{i}}(\tilde{x},\mathbf{0})|\Gamma_{1}(t)\right. \\
	+\left. \left(|\frac{\partial \mathcal{F}}{\partial \bar{\rho}_{l}}(\tilde{x},\mathbf{0})|
	+|\frac{\partial \mathcal{F}}{\partial \bar{u}_{l}}(\tilde{x},\mathbf{0})|+|\frac{\partial \mathcal{F}}{\partial \bar{p}_{l}}(\tilde{x},\mathbf{0})|\right)\left(C_{l}\epsilon\Gamma_{2}(t) + C_{L}\epsilon\zeta^{N}\right)\right).
\end{multline}
\par Then, for fixed $(t,x)$, one can integrate \eqref{d10} along the characteristic $X_{3}(\tau;t,x)$. Without loss of generality, we assume that $X_{3}(\tau;t,x)$ intersects with $\chi$ (rather than $\chi^{(\mathcal{T})}$) at the point $(\tau_{3}, \chi(\tau_{3}))$. Applying the very similar argument as \eqref{c55}--\eqref{cauchy}, and further noting \eqref{d2}, one can conclude
\begin{equation}
	\Gamma_{1}(t) \le (1+C\epsilon +C\mathcal{E} )\left(\alpha_{3}\theta_{d}+\frac{\alpha_{3}}{\alpha_{2}}\right)\alpha C_{S,1}\epsilon\zeta^{N} +CC_{L}\epsilon\zeta^{N} + \alpha CC_{S,1}\epsilon^{2}\zeta^{N}, \quad \forall t\in [t_{0},t_{*}].
\end{equation}
Thus, one can choose $\zeta$ satisfying
\begin{eqnarray*}
	(1+C\epsilon +C\mathcal{E} )\left(\alpha_{3}\theta_{d}+\frac{\alpha_{3}}{\alpha_{2}}\right)\alpha < \zeta < 1,
\end{eqnarray*}
such that for large enough $C_{S,1}$, 
\begin{equation}\label{Gamma1}
	\Gamma_{1}(t) \le C_{S,1}\epsilon\zeta^{N+1}, \quad \forall t\in [t_{0},t_{*}].
\end{equation}
The analysis for $\breve{\Upsilon}_{2}$ and $\breve{\Upsilon}_{3}$ is similar or simpler. 
\par Moreover, one can choose $\zeta$ and $\alpha$ satisfying
\begin{eqnarray*}
	\alpha\left(1+2\bar{C}\mathcal{E}\mathcal{T}_{0}\right)^{-1} < \zeta <1,
\end{eqnarray*}
and then the estimate for $\Gamma_{2}(t)$ in \eqref{bootstrapGoal} follows from \eqref{smallZeta}. Finally, the estimate for $\Gamma_{3}(t)$ in \eqref{bootstrapGoal} can be derived directly from \eqref{shockposition} and \eqref{Gamma1} by using Hadamard's formula. The proof of \Cref{t2} is now complete.

\section*{Acknowledgments}
Qu and Wang are supported in part by NSFC Grant No. 12431007 and 62588101. Yu and Zhang are supported by NSFC Grant No. 12271310 and Natural Science Foundation of Shandong Province ZR2022MA088. The authors would like to thank Professor Hairong Yuan for his discussion and help.

\small
\begin{spacing}{1}
	\bibliographystyle{abbrv}
	\bibliography{myShockBibs}
\end{spacing}
\end{document}